\documentclass[11pt,letterpaper]{amsart}
\usepackage{geometry}
\usepackage{amsmath,amssymb,amsthm,mathtools}
\usepackage{microtype}
\allowdisplaybreaks
\usepackage[hidelinks]{hyperref}
\newtheorem{theorem}{Theorem}[section]
\newtheorem{proposition}[theorem]{Proposition}
\newtheorem{lemma}[theorem]{Lemma}
\newtheorem{corollary}[theorem]{Corollary}
\theoremstyle{remark}
\newtheorem{remark}[theorem]{Remark}
\numberwithin{equation}{section}

\title{On Dolbeault's conjecture}
\author[J. Wei]{Juncheng Wei}
\address[J. Wei]{Department of Mathematics \\ The Chinese University of Hong Kong \\ Shatin, N.T. \\ Hong Kong}
\email{wei@math.cuhk.edu.hk}
\author[Z. Yan]{Zetian Yan}
\address[Z. Yan]{Department of Mathematics \\ The Chinese University of Hong Kong \\ Shatin, N.T. \\ Hong Kong}
\email{zetianyan@cuhk.edu.hk}
\author[Z. Ye]{Zikai Ye}
\address{Department of Mathematics\\ Uppsala University\\ 751 06 Uppsala  \\ Sweden
}
\email{zikai.ye@math.uu.se}
\date{}

\begin{document}
\begin{abstract}
We prove sharp improved Sobolev inequalities on $\mathbb S^n$, $n\ge3$, for antipodally even functions and for axially symmetric functions whose critical density $|u|^{2n/(n-2)}$ has zero barycenter. In both cases, the Dirichlet coefficient improves
by the optimal factor $2^{-2/n}$, while the coefficient of the $L^2$ term remains one. Equality holds precisely for constant functions. The even case resolves a conjecture of Dolbeault (2021).
\end{abstract}

\maketitle

\section{Introduction}

Let $n\ge3$, and equip the unit sphere $\mathbb S^n$ with its
canonical metric and normalized volume measure $
 d\sigma=\frac{dV_{\mathbb S^n}}{\operatorname{Vol}(\mathbb S^n)}$.
All norms and scalar products on the sphere are taken with respect
to $d\sigma$. We call a function $u$ antipodally even if
$u(-\xi)=u(\xi)$ almost everywhere. In
\cite[Section~4]{Dolbeault}, Dolbeault proposed the following
improvement of the critical Sobolev inequality.

\medskip
\noindent\textbf{Conjecture.}
For every real-valued antipodally even function
$u\in H^1(\mathbb S^n)$, it holds
\begin{equation}\label{eq:local-target-1}
 \|u\|_{2n/(n-2)}^2\le
 \frac{4\,2^{-2/n}}{n(n-2)}\|\nabla u\|_2^2+\|u\|_2^2.
\end{equation}
\medskip

With $2^{-2/n}$ replaced by $1$, \eqref{eq:local-target-1} is the spherical form of the classical sharp Sobolev inequality, whose Euclidean version was established by Aubin \cite{Aubin} and Talenti \cite{Talenti}. This inequality belongs to a larger family of conformally invariant inequalities. Lieb's sharp Hardy--Littlewood--Sobolev inequality \cite{Lieb1983} yields,
by duality, sharp Sobolev inequalities of fractional order, and Beckner \cite{Beckner} obtained their spherical formulation.
The associated conformally covariant operators include the conformal Laplacian, the higher-order GJMS operators of
Graham, Jenne, Mason, and Sparling \cite{GJMS}, and their fractional-order analogues. At the critical-order limit,
Beckner's theory also yields a higher-dimensional extension of the Moser--Trudinger--Onofri inequality.

A different refinement is obtained by imposing symmetry or moment constraints. In the second-order case considered here, Aubin \cite{Aubin1979} showed that, for every $\varepsilon>0$,
there exists $C_\varepsilon>0$ such that
\begin{equation}
 \|u\|_{2n/(n-2)}^2
 \le
 \frac{4(2^{-2/n}+\varepsilon)}{n(n-2)}
 \|\nabla u\|_2^2
 +C_\varepsilon\|u\|_2^2
\end{equation}
for every $u\in H^1(\mathbb S^n)$ satisfying
\begin{equation}\label{balanced}
 \int_{\mathbb S^n}\xi_i|u|^{2n/(n-2)}\,d\sigma=0,
 \qquad i=1,\ldots,n+1;
\end{equation}
see also \cite[Section~1]{HangWang}. The natural sharp version of this estimate would take $\varepsilon=0$ and
$C_\varepsilon=1$.

Every antipodally even function satisfies \eqref{balanced}, so Dolbeault's conjecture is the restriction of this proposed sharp inequality to antipodally even functions. The factor $2^{-2/n}$ cannot be decreased, as shown by two equally weighted Sobolev extremals concentrating at antipodal points; see
Remark~\ref{rem:sharpness} and Section~\ref{sec:axial-9}. For higher-order moment constraints,
Hang and Wang \cite{HangWang} and Putterman \cite{Putterman}
identified optimal leading coefficients in several cases. These developments make it natural to ask whether the optimal leading coefficient can be attained while keeping the coefficient of the lower-order $L^2$ term equal to one.

A closely related question arises at the logarithmic endpoint.
The Chang--Yang conjecture concerns the sharp improvement of the Moser--Trudinger--Onofri inequality under a zero-barycenter constraint on the exponential density. Gui and Moradifam
\cite{GuiMoradifam} resolved the conjecture on $\mathbb S^2$ using the sphere covering inequality. In a recent preprint, Gui, Li, Wei, and Ye \cite{GuiLiWeiYe} established the higher-dimensional version for every integer dimension at
least three, with the optimal improvement factor $1/2$ in the quadratic energy. Their proof uses rigidity restricted to stable constrained critical points, without requiring a classification of all critical points. Recently, the classification of all critical points are established by Gui, Hu, Li, Wei, Xie and Ye \cite{GHLWXY}.

Here we return to the finite critical exponent and prove the
sharp second-order inequality for antipodally even functions
and for axially-symmetric functions satisfying \eqref{balanced}. Set
\[
 p=\frac{2n}{n-2},\qquad
 \alpha=2^{-2/n},\qquad
 P_2=-\frac4{n(n-2)}\Delta_{\mathbb S^n}+\mathrm{id}.
\]
Thus $P_2$ is the conformal Laplacian normalized by $P_2 1=1$,
and the classical sharp Sobolev inequality on the sphere is
\begin{equation}\label{eq:ordinary}
 \|u\|_p^2
 \le
 \langle u,P_2u\rangle
 =
 \|u\|_2^2+\frac4{n(n-2)}\|\nabla u\|_2^2,
 \qquad u\in H^1(\mathbb S^n);
\end{equation}
see \cite{Beckner,Case}. The quadratic form
$\langle u,P_2u\rangle$ defines the square of a norm equivalent
to the usual $H^1$ norm. Equality in \eqref{eq:ordinary} holds
precisely for constant multiples of conformal factors;
see \cite{Beckner,FrankLieb}. Our first result establishes
Dolbeault's conjecture.

\begin{theorem}\label{thm:main}
Let $n\ge3$. Every real-valued antipodally even function
$u\in H^1(\mathbb S^n)$ satisfies
\begin{equation}\label{eq:main}
 \|u\|_p^2\le\alpha\langle u,P_2u\rangle
 +(1-\alpha)\|u\|_2^2
 =\frac{4\alpha}{n(n-2)}\|\nabla u\|_2^2+\|u\|_2^2.
\end{equation}
Equality holds if and only if $u$ is constant.
The coefficient $\alpha=2^{-2/n}$ is optimal.
\end{theorem}

Theorem~\ref{thm:main} confirms Dolbeault's conjecture.
Our second result establishes the same inequality for axial
functions under the balanced condition. After a rotation, write
$x=\xi_{n+1}$. A function is axially-symmetric if it depends only on $x$.

\begin{theorem}\label{thm:axial}
Let $n\ge3$ be an integer. If $u\in H^1(\mathbb S^n)$ is
a real-valued axially-symmetric function satisfying
\[
 \int_{\mathbb S^n}x|u|^p\,d\sigma=0,
\]
then
\begin{equation}\label{ax:eq:1}
 \|u\|_p^2\le\|u\|_2^2+
 \frac{4\alpha}{n(n-2)}\|\nabla u\|_2^2.
\end{equation}
Equality holds if and only if $u$ is constant.
The coefficient $\alpha=2^{-2/n}$ is optimal.
\end{theorem}

The proofs use different consequences of the two symmetry
assumptions. For Theorem~\ref{thm:main}, an antipodally even
function descends to $\mathbb{RP}^n$. For
$M=\mathbb S^n$ or $\mathbb{RP}^n$, let $I_M$ denote the
isoperimetric profile with both volume and perimeter normalized
by $\operatorname{Vol}(M)$. Viana's classification
\cite[Corollary~1.3]{Viana} identifies the isoperimetric regions
in $\mathbb{RP}^n$ as tubular neighborhoods of totally geodesic
projective subspaces. Comparing their perimeter functions with
the spherical profile, Lemma~\ref{lem:profile} gives
\[
 I_{\mathbb{RP}^n}(v)>2^{1/n}I_{\mathbb S^n}(v),
 \qquad 0<v<1.
\]
The factor $2^{1/n}$ is optimal, as follows by letting $v\to0$.
For a nonnegative function $f$ on $\mathbb{RP}^n$, let $f^*$
be its equimeasurable rearrangement on $\mathbb S^n$ whose
superlevel sets are concentric spherical caps. The coarea formula
and the profile comparison yield
\[
 \int_{\mathbb S^n}|\nabla f^*|^2\,d\sigma
 \le2^{-2/n}\int_{\mathbb{RP}^n}|\nabla f|^2\,d\sigma,
 \qquad f\in H^1(\mathbb{RP}^n),\quad f\ge0,
\]
with equality only for constants. Here the measure on projective
space also has total mass one. Combining this estimate with
\eqref{eq:ordinary} proves Theorem~\ref{thm:main}.

The proof of Theorem~\ref{thm:axial} is variational. An axially-symmetric function need not be even, so the preceding argument on $\mathbb{RP}^n$ does not apply. We normalize the $L^p$ norm to one and first address the possible loss of compactness. Axial symmetry gives local compactness away from the two poles, so any defect in the $L^p$ norm is supported at those poles.
We therefore allow polar atoms in the variational problem and prove that the relaxed infimum is attained. By Propositions~\ref{ax:prop:2.2} and~\ref{ax:prop:3.1},
every relaxed minimizer is either a pair of equal polar atoms with energy one or a smooth positive function.

For a smooth minimizer, the Euler--Lagrange equation is
\[
 u-\frac{4\alpha}{n(n-2)}\Delta_{\mathbb S^n}u
 =(a+bx)u^{p-1},
 \qquad 0<a\le1,
\]
where $a$ is its energy and $b$ is the multiplier associated
with the barycenter constraint. We retain $b$ throughout the
argument. The conformal variation generated by $\nabla x$
preserves the critical mass but is not tangent to the
barycenter constraint. We correct it by a multiple of $xu$
determined by the second moment:
\[
 \begin{gathered}
 m=\int_{\mathbb S^n}x^2u^p\,d\sigma,\\
 \phi
 =
 \langle\nabla x,\nabla u\rangle
 -\frac{n-2}{2}xu
 +\frac{(n-2)(1-m)}{2nm}\,xu.
 \end{gathered}
\]
This direction satisfies both linearized constraints. Writing
$Q$ for the constrained second-variation form, stability
therefore requires $Q(\phi,\phi)\ge0$.

The exact identity \eqref{ax:eq:15} expresses $Q(\phi,\phi)$ as a negative multiple of $\|\phi\|_2^2$, together with terms involving the total mass and second moment of $u^2$ and the mixed quantity $b\int_{\mathbb S^n}x^3u^p\,d\sigma$. The negative square term $\|\phi\|_2^2$ alone does not determine the sign.
The Euler--Lagrange identities \eqref{ax:eq:19} and \eqref{ax:eq:22} yield two equivalent
expressions for the second variation. The gradient-ratio
expression \eqref{ax:eq:27} allows us to control $b\int_{\mathbb S^n}x^3u^p\,d\sigma$ using the moment bounds in Lemma~\ref{ax:lem:4.2},
whereas \eqref{ax:eq:25} allows us to eliminate it through
the weighted Sobolev identity \eqref{ax:eq:32}.
Using the identities \eqref{ax:eq:32} and \eqref{ax:eq:47}, we apply the
classical sharp Sobolev inequality to $\sqrt{1\pm x}\,u$
and $(\tau\pm x)u$, with suitable $\tau\ge1$.
Hermite interpolation and moment comparisons provide the
required lower bounds for their critical norms; see
\eqref{ax:eq:33}, \eqref{ax:eq:37}, and
Lemma~\ref{ax:lem:8.1}.

We divide the instability argument into cases according to the value of
\[
 m=\int_{\mathbb S^n}x^2u^p\,d\sigma,
\]
the second moment of the critical density.
The following table summarizes the expressions of $Q(\phi,\phi)$, the bounds on the second variation, and the coefficient estimates used in each range of $m$.

\begin{center}
\begin{tabular}{c|c|l}
\hline
Dimension & Range of $m$ & Main argument \\
\hline
$n\ge3$
&
$\displaystyle \frac1{n+1}<m\le\frac2{n+2}$
&
expression \eqref{ax:eq:27} and coefficients bound \eqref{ax:eq:fngn}
\\[4pt]
\hline
$n\ge6$
&
$\displaystyle \frac2{n+2}<m\le\frac2n$
&
expression \eqref{ax:eq:27} and coefficients bound \eqref{ax:eq:fngnn6}
\\[4pt]
$n\ge6$
&
$\displaystyle \frac2n<m<\frac{n+3}{3(n+1)}$
&
expression \eqref{ax:eq:27}, comparision \eqref{ax:eq:41}, and coefficients bound \eqref{ax:eq:fngnhn}
\\[4pt]
$n\ge6$
&
$\displaystyle m\ge\frac{n+3}{3(n+1)}$
&
excluded by \eqref{ax:eq:45}
\\[4pt]
\hline
$n=4,5$
&
$\displaystyle \frac2{n+2}<m<1$
&
expression \eqref{ax:eq:25}, lower bound \eqref{ax:eq:low-dimensional-second-variation}
and coefficients bound \eqref{ax:eq:A2}
\\[4pt]
\hline
$n=3$
&
$\displaystyle \frac25<m<\frac47$
&
expression \eqref{ax:eq:25} and coefficients bound \eqref{ax:eq:48}
\\[4pt]
$n=3$
&
$\displaystyle \frac47\le m<1$
&
expression \eqref{ax:eq:25}, lower bound \eqref{ax:eq:low-dimensional-second-variation}
and coefficients bound \eqref{ax:eq:A2}
\\
\hline
\end{tabular}
\end{center}

These estimates prove $Q(\phi,\phi)<0$ for every non-constant
smooth solution under consideration; see
Theorem~\ref{ax:thm:4.1}. Hence every smooth minimizer is
constant. Together with the purely atomic alternative, this
shows that the relaxed infimum is one and proves the inequality under axial symmetry. The same argument gives the equality
classification, while two antipodal concentrating bubbles
establish the optimality of the gradient coefficient.

For the sharp improved Sobolev inequalities without symmetry assumption, recently, we are able to establish it for $n\geq 6$.

The paper is organized as follows.
Section~\ref{sec:even} proves the inequality for antipodally
even functions. Section~\ref{sec:axial-1} establishes the
compactness reduction for the axial problem.
Section~\ref{sec:axial-4} develops the variational identities
and moment estimates used in the instability argument of
Section~\ref{sec:axial-instability}.
Section~\ref{sec:axial-9} completes the axial proof, including
the equality cases and sharpness.
The dimension-dependent coefficient estimates are collected
in Lemmas~\ref{ax:lem:A1}--\ref{ax:lem:A4} and Lemma~\ref{ax:lem:D1}.

\subsection*{Acknowledgments}
Juncheng Wei was supported by National R\&D Program of China (Grant No. 2022YFA1005602), and Hong Kong General Research Fund (No. 14303125) ``On Fujita equation in the critical or supercritical regime''. Zetian Yan is substantially supported by research grants from the Research Grants Council of the Hong Kong Special Administrative Region, China [Project No.: CUHK 14307123, CUHK 14306624 and CUHK 14308325].

\subsection*{Disclosure on AI assistance}

The authors used AI-assisted tools, principally ChatGPT/Codex, for
brainstorming, literature searches and manuscript editing.
AI output was not treated as mathematical authority. The authors
independently worked and verified all theorem statements, proofs,
calculations, constants, citations, and final prose, and take full
responsibility for the contents of the paper.

\section{The antipodally even case}\label{sec:even}

In this section, we are going to prove Theorem~\ref{thm:main}.
We regard an antipodally even function on $\mathbb S^n$ as a function on $\mathbb{RP}^n$ and compare its Dirichlet energy
with that of its spherical rearrangement. The required estimate follows from a strict comparison of the isoperimetric profiles,
based on Viana's classification of isoperimetric regions in $\mathbb{RP}^n$. Combining this estimate with the sharp Sobolev inequality on $\mathbb S^n$ yields \eqref{eq:local-target-1} together with its equality case.

We first define the isoperimetric profile. Let $M=\mathbb S^n$ or $\mathbb{RP}^n$ be equipped with the canonical metric of sectional curvature one. The normalized volume measure and perimeter are given by
\[
 d\sigma_M=\frac{dV_M}{\operatorname{Vol}(M)},\qquad
 \operatorname{Per}_M(E)
 =\frac{\mathcal H^{n-1}(\partial^*E)}{\operatorname{Vol}(M)},
\]
where $E\subset M$ is a set of finite perimeter and $\partial^*E$ denotes its reduced boundary. The isoperimetric profile of $M$ is
\begin{equation}
 I_M(v)=\inf\{\operatorname{Per}_M(E):\sigma_M(E)=v\},
 \qquad 0\le v\le1,
\end{equation}
where the infimum is taken over sets of finite perimeter. Taking complements yields $I_M(v)=I_M(1-v)$.
Since the canonical covering
$\pi:\mathbb S^n\to\mathbb{RP}^n$ is a local isometry of degree two,
\begin{equation}\label{eq:cover-normalization}
 \sigma_{\mathbb S^n}(\pi^{-1}E)=\sigma_{\mathbb{RP}^n}(E),\qquad
 \operatorname{Per}_{\mathbb S^n}(\pi^{-1}E)
 =\operatorname{Per}_{\mathbb{RP}^n}(E).
\end{equation}
We omit the subscript in $d\sigma_M$ when the domain is clear.

By Viana's classification \cite[Corollary 1.3]{Viana}, every
isoperimetric region in $\mathbb{RP}^n$ is, up to a set of measure
zero, a tubular neighborhood of a totally geodesic projective subspace. For $1\le j\le n$ and $0<r<1$, define
\[
 X_j(\xi)=\sum_{i=1}^j\xi_i^2,\qquad
 E_{j,r}=\{\xi\in\mathbb S^n:X_j(\xi)<r\}.
\]
It is easy to verify that each $E_{j,r}$ is antipodally invariant, and $\pi(E_{j,r})$ is
the tubular neighborhood of radius $\arcsin\sqrt{r}$ about $\mathbb{RP}^{n-j}$. Moreover, $ \partial E_{j,r}$ is given by
\[
 \partial E_{j,r}
 =\mathbb S^{j-1}(\sqrt{r})
  \times\mathbb S^{n-j}(\sqrt{1-r}).
\]
In particular, in the coordinates
\[
 \xi=(\sqrt{r}\,\eta,\sqrt{1-r}\,\zeta),\qquad
 \eta\in\mathbb S^{j-1},\quad \zeta\in\mathbb S^{n-j},
\]
the spherical volume element is
\[
 dV_{\mathbb S^n}
 =\frac12 r^{\frac{j-2}{2}}(1-r)^{\frac{n-j-1}{2}}
 \,dr\,dV_{\mathbb S^{j-1}}\,dV_{\mathbb S^{n-j}}.
\]
Let $\rho_j$ denote the probability density of $X_j$ under
$d\sigma_{\mathbb S^n}$, and let $v_j(r)$ denote the normalized
volume of $E_{j,r}$. The preceding volume formula gives
\begin{equation}\label{eq:beta-density}
 \rho_j(r)
 =\frac{r^{\frac{j-2}{2}}(1-r)^{\frac{n-j-1}{2}}}
 {\mathrm B\bigl(\frac j2,\frac{n+1-j}{2}\bigr)},
 \qquad
 v_j(r)=\sigma_{\mathbb S^n}(E_{j,r})
 =\int_0^r\rho_j(t)\,dt,
\end{equation}
where $\mathrm B$ denotes the beta function.
Since $|\nabla X_j|^2=4X_j(1-X_j)$, the coarea formula gives
the normalized perimeter
\[
 P_j(r):=\operatorname{Per}_{\mathbb S^n}(E_{j,r})
 =2\sqrt{r(1-r)}\,\rho_j(r).
\]
The function $v_j:(0,1)\to(0,1)$ is strictly increasing and onto.
We may therefore express the perimeter as a function of volume by
\[
 J_j(v)=P_j\bigl(v_j^{-1}(v)\bigr),\qquad 0<v<1.
\]
Viana's classification and \eqref{eq:cover-normalization} imply
\begin{equation}\label{eq:profile-envelope}
 I_{\mathbb{RP}^n}(v)=\min_{1\le j\le n}J_j(v),
 \qquad 0<v<1.
\end{equation}
Taking complements and interchanging the two factors yields
\[
 J_j(v)=J_{n+1-j}(1-v).
\]

We next establish a strict comparison between the isoperimetric
profiles of $\mathbb{RP}^n$ and $\mathbb S^n$.

\begin{lemma}\label{lem:profile}
For $0<v<1$,
\begin{equation}\label{eq:profile}
 I_{\mathbb{RP}^n}(v)>2^{1/n}I_{\mathbb S^n}(v).
\end{equation}
\end{lemma}

\begin{proof}
By \eqref{eq:profile-envelope}, it suffices to prove that
\[
 J_j(v)>2^{1/n}I_{\mathbb S^n}(v)
 \qquad\text{for }1\le j\le n,\quad 0<v<1.
\]
Let $r=v_j^{-1}(v)$. Since $v_j'(r)=\rho_j(r)$,
differentiation with respect to $v$ gives
\begin{align}
 J_j'(v)
 &=(j-1)\sqrt{\frac{1-r}{r}}
   -(n-j)\sqrt{\frac{r}{1-r}},\notag\\
 -J_j(v)J_j''(v)
 &=\frac{j-1}{r}+\frac{n-j}{1-r}.
 \label{eq:profile-derivatives}
\end{align}

We also record the corresponding identity for the spherical
isoperimetric profile. Since spherical caps are isoperimetric,
a cap of radius $\theta\in(0,\pi)$ satisfies
\[
 v=\frac{|\mathbb S^{n-1}|}{|\mathbb S^n|}
       \int_0^\theta\sin^{n-1}t\,dt,\qquad
 I_{\mathbb S^n}(v)
 =\frac{|\mathbb S^{n-1}|}{|\mathbb S^n|}
       \sin^{n-1}\theta.
\]
In particular, $dv/d\theta=I_{\mathbb S^n}(v)$, and hence
\[
 I_{\mathbb S^n}'(v)=(n-1)\cot\theta,\qquad
 I_{\mathbb S^n}''(v)
 =-\frac{n-1}{I_{\mathbb S^n}(v)\sin^2\theta}.
\]
Therefore
\begin{equation}\label{eq:spherical-profile}
 -I_{\mathbb S^n}(v)I_{\mathbb S^n}''(v)
 =(n-1)+\frac{I_{\mathbb S^n}'(v)^2}{n-1}.
\end{equation}

Suppose first that $2\le j\le n-1$.
By \eqref{eq:beta-density}, there exist positive constants
$c_j$ and $d_j$, depending only on $n$ and $j$, such that
\[
 J_j(v)\sim c_j v^{(j-1)/j}
 \quad\text{as }v\to0,\qquad
 J_j(v)\sim d_j(1-v)^{(n-j)/(n+1-j)}
 \quad\text{as }v\to1.
\]
Here $f(v)\sim g(v)$ means that $f(v)/g(v)\to1$
in the indicated limit.
The spherical isoperimetric profile satisfies
\[
 I_{\mathbb S^n}(v)\sim c\,v^{(n-1)/n}
 \quad\text{as }v\to0
\]
for some constant $c>0$ depending only on $n$.
Since
\[
 \frac{j-1}{j}<\frac{n-1}{n},\qquad
 \frac{n-j}{n+1-j}<\frac{n-1}{n},
\]
these asymptotics and the symmetry
$I_{\mathbb S^n}(v)=I_{\mathbb S^n}(1-v)$ imply
\[
 \frac{J_j(v)}{I_{\mathbb S^n}(v)}\to\infty
 \qquad\text{as }v\to0\text{ or }v\to1.
\]

Thus this ratio attains its minimum at some $v_0\in(0,1)$.
Set
\[
 R=\frac{J_j(v_0)}{I_{\mathbb S^n}(v_0)},\qquad
 r_0=v_j^{-1}(v_0).
\]
At $v_0$,
\[
 J_j'(v_0)=R I_{\mathbb S^n}'(v_0),\qquad
 J_j''(v_0)\ge R I_{\mathbb S^n}''(v_0).
\]
Using \eqref{eq:profile-derivatives} and
\eqref{eq:spherical-profile}, we obtain
\begin{align*}
 \frac{j-1}{r_0}+\frac{n-j}{1-r_0}=-J_j(v_0)J_j''(v_0)&\le -R^2I_{\mathbb S^n}(v_0)I_{\mathbb S^n}''(v_0)=R^2(n-1)+\frac{J_j'(v_0)^2}{n-1}.
\end{align*}

On the other hand, the formula for $J_j'$ gives
\[
 \frac{j-1}{r_0}+\frac{n-j}{1-r_0}
 -\frac{J_j'(v_0)^2}{n-1}
 =(n-1)+\frac{(j-1)(n-j)}{(n-1)r_0(1-r_0)}.
\]
    
Consequently,
\[
 R^2
 \ge1+\frac{(j-1)(n-j)}{(n-1)^2r_0(1-r_0)}
 \ge1+\frac{4(n-2)}{(n-1)^2}
 >2^{2/n}.
\]
Here we used $(j-1)(n-j)\ge n-2$ and $r_0(1-r_0)\le1/4$.
The last inequality follows from
\[
 2^{2/n}<1+\frac3n
 \le1+\frac{4(n-2)}{(n-1)^2},
 \qquad n\ge3,
\]
where the strict inequality follows from the strict concavity
of $t\mapsto t^{1/n}$.

It remains to consider $j=1$ and $j=n$.
The region $E_{n,r}$ consists of two antipodal caps, while
$E_{1,r}$ is an equatorial band. Thus
\begin{equation}
 J_n(v)=2I_{\mathbb S^n}(v/2),\qquad
 J_1(v)=2I_{\mathbb S^n}((1-v)/2).
\end{equation}
By the symmetry of $I_{\mathbb S^n}$ and the identity
$J_1(v)=J_n(1-v)$, it suffices to consider $0<v\le1/2$.
Since $I_{\mathbb S^n}$ is increasing on $[0,1/2]$,
\[
 J_1(v)\ge J_n(v).
\]

We claim that $v^{-(n-1)/n}I_{\mathbb S^n}(v)$ is strictly
decreasing on $(0,1/2]$. Indeed, for $0<\theta\le\pi/2$,
\[
 n\cos\theta\int_0^\theta\sin^{n-1}t\,dt
 <n\int_0^\theta\sin^{n-1}t\cos t\,dt
 =\sin^n\theta.
\]

Together with the formulas for $v$, $I_{\mathbb S^n}(v)$,
and $I_{\mathbb S^n}'(v)$, this gives
\[
 vI_{\mathbb S^n}'(v)
 <\frac{n-1}{n}I_{\mathbb S^n}(v),
\]
which proves the claim. It follows that
\[
 J_n(v)=2I_{\mathbb S^n}(v/2)
 >2^{1/n}I_{\mathbb S^n}(v),
 \qquad 0<v\le1/2.
\]
Hence the required inequality holds for every $J_j$, which completes the proof.
\end{proof}

The constant $2^{1/n}$ in \eqref{eq:profile} is optimal.
Indeed,
\[
 I_{\mathbb{RP}^n}(v)\le J_n(v)=2I_{\mathbb S^n}(v/2).
\]
Together with Lemma~\ref{lem:profile} and the asymptotic behavior
of $I_{\mathbb S^n}$ at $v=0$, it implies
\[
 \lim_{v\to0}
 \frac{I_{\mathbb{RP}^n}(v)}{I_{\mathbb S^n}(v)}
 =2^{1/n}.
\]

We now apply this comparison to spherical rearrangement.
For a nonnegative measurable function $f$ on $\mathbb{RP}^n$,
define its distribution function by
\[
 \mu_f(t)=\sigma_{\mathbb{RP}^n}(\{f>t\}),
 \qquad t\ge0,
\]
and its decreasing rearrangement on $(0,1)$ by
\[
 f^\sharp(v)=\inf\{t\ge0:\mu_f(t)\le v\},
 \qquad 0<v<1.
\]

Fix a point $e\in\mathbb S^n$ and define the volume function of the geodesic ball $B_{\theta}(e)$ by
\[
 V(\theta)=\sigma_{\mathbb S^n}(B_\theta(e))
 =\frac{|\mathbb S^{n-1}|}{|\mathbb S^n|}
       \int_0^\theta\sin^{n-1}t\,dt,
 \qquad 0\le\theta\le\pi.
\]

The spherical rearrangement of $f$ is defined by
\begin{equation}\label{eq:rearrangement-definition}
 f^*(\xi)
 =f^\sharp\bigl(V(d_{\mathbb S^n}(e,\xi))\bigr),
 \qquad \xi\in\mathbb S^n\setminus\{e,-e\}.
\end{equation}

Its values at the two poles may be chosen arbitrarily.
Up to sets of measure zero, each superlevel set $\{f^*>t\}$
is a cap centered at $e$ with normalized volume $\mu_f(t)$.
Thus
\[
 \sigma_{\mathbb S^n}(\{f^*>t\})
 =\sigma_{\mathbb{RP}^n}(\{f>t\}),
 \qquad t\ge0,
\]
so $f$ and $f^*$ are equimeasurable and have the same
$L^r$ norms for $1\le r\le\infty$.
In general, $f^*$ is not antipodally even.
The Hardy--Littlewood inequality and preservation of the
$L^2$ norm give
\begin{equation}\label{eq:rearrangement-contraction}
 \|f^*-g^*\|_{L^2(\mathbb S^n)}
 \le\|f-g\|_{L^2(\mathbb{RP}^n)}
\end{equation}
for all nonnegative $f,g\in L^2(\mathbb{RP}^n)$.
For real-valued functions, we define
$f^\sharp=(|f|)^\sharp$ and $f^*=(|f|)^*$. The isoperimetric comparison in Lemma~\ref{lem:profile} yields
the following Dirichlet energy estimate for spherical rearrangement.

\begin{lemma}\label{lem:rearrangement-energy}
Let $f\in H^1(\mathbb{RP}^n)$ be nonnegative.
Then $f^*\in H^1(\mathbb S^n)$ and
\begin{equation}\label{eq:dirichlet-transfer}
 \int_{\mathbb S^n}|\nabla f^*|^2\,d\sigma
 \le2^{-2/n}
 \int_{\mathbb{RP}^n}|\nabla f|^2\,d\sigma,
\end{equation}
with equality if and only if $f$ is constant.
\end{lemma}

\begin{proof}
We first assume that $f$ is a nonnegative smooth Morse function.
On each regular level set, let
\[
 dA=\frac{d\mathcal H^{n-1}}{\operatorname{Vol}(\mathbb{RP}^n)}
\]
denote the normalized area measure. First, by the Cauchy--Schwarz inequality, we have
\[
 \operatorname{Per}_{\mathbb{RP}^n}(\{f>t\})^2
 \le
 \left(
 \int_{\{f=t\}}|\nabla f|\,dA
 \right)
 \left(
 \int_{\{f=t\}}\frac{1}{|\nabla f|}\,dA
 \right).
\]
For each regular value $t$, substituting $\int_{\{f=t\}}\frac{1}{|\nabla f|}\,dA$ by $-\mu_f'(t)$ yields
\[
 \int_{\{f=t\}}|\nabla f|\,dA
 \ge
 \frac{
 \operatorname{Per}_{\mathbb{RP}^n}(\{f>t\})^2
 }{-\mu_f'(t)}.
\]

Combining the fact
\[
 \int_{\mathbb{RP}^n}|\nabla f|^2\,d\sigma
 =
 \int_{\min f}^{\max f}
 \left(
 \int_{\{f=t\}}|\nabla f|\,dA
 \right)dt
\]
with 
\[
 \operatorname{Per}_{\mathbb{RP}^n}(\{f>t\})
 \ge
 I_{\mathbb{RP}^n}(\mu_f(t)),
\]
we obtain
\[
 \int_{\mathbb{RP}^n}|\nabla f|^2\,d\sigma
 \ge
 \int_{\min f}^{\max f}
 \frac{
 I_{\mathbb{RP}^n}(\mu_f(t))^2
 }{-\mu_f'(t)}\,dt.
\]

Since $f^\sharp$ is the inverse of the decreasing distribution function
$\mu_f$ away from critical values, the change of variables
$v=\mu_f(t)$ gives
\[
 (f^\sharp)'(v)=\frac{1}{\mu_f'(t)}
\]
and
\begin{align*}
 \int_{\min f}^{\max f}
 \frac{
 I_{\mathbb{RP}^n}(\mu_f(t))^2
 }{-\mu_f'(t)}\,dt
 &=
 \int_0^1
 I_{\mathbb{RP}^n}(v)^2
 \frac{1}{|\mu_f'(f^\sharp(v))|^2}\,dv=
 \int_0^1
 I_{\mathbb{RP}^n}(v)^2
 |(f^\sharp)'(v)|^2\,dv.
\end{align*}
Consequently,
\begin{equation}\label{eq:quantile-projective}
 \int_{\mathbb{RP}^n}|\nabla f|^2\,d\sigma
 \ge
 \int_0^1 I_{\mathbb{RP}^n}(v)^2
 |(f^\sharp)'(v)|^2\,dv.
\end{equation}

Since $V'(\theta)=I_{\mathbb S^n}(V(\theta))$,
the definition of $f^*$ gives
\begin{equation}\label{eq:quantile-spherical}
 \int_{\mathbb S^n}|\nabla f^*|^2\,d\sigma
 =
 \int_0^1 I_{\mathbb S^n}(v)^2
             |(f^\sharp)'(v)|^2\,dv.
\end{equation}
Combining these identities with Lemma~\ref{lem:profile}
proves \eqref{eq:dirichlet-transfer} for smooth Morse functions.

Now let $f$ be a nonnegative function in $H^1(\mathbb{RP}^n)$.
We choose a sequence of nonnegative smooth Morse functions $\{f_\ell\}$ such that
$f_\ell\to f$ in $H^1$.
Such a sequence can be obtained by smooth nonnegative
approximation, followed by small perturbations and
constant shifts tending to zero.
By \eqref{eq:rearrangement-definition} and
\eqref{eq:rearrangement-contraction},
\[
 \|f_\ell^\sharp-f^\sharp\|_{L^2(0,1)}
 =
 \|f_\ell^*-f^*\|_{L^2(\mathbb S^n)}
 \to0.
\]

We first prove \eqref{eq:quantile-projective} for $f$.
Fix $0<\delta<1/2$.
By \eqref{eq:profile-envelope}, $I_{\mathbb{RP}^n}$
is continuous and positive on $[\delta,1-\delta]$.
Applying \eqref{eq:quantile-projective} to $f_\ell$ shows
that $\{f_\ell^\sharp\}$ is bounded in
$H^1([\delta,1-\delta])$.
The $L^2$ convergence above then implies
\[
 f^\sharp\in H^1([\delta,1-\delta]),\qquad
 f_\ell^\sharp\rightharpoonup f^\sharp
 \quad\text{in }H^1([\delta,1-\delta]).
\]
Weak lower semicontinuity and the strong $H^1$ convergence
of $f_\ell$ to $f$ yield
\[
 \int_\delta^{1-\delta}
 I_{\mathbb{RP}^n}(v)^2 |(f^\sharp)'(v)|^2\,dv
 \le
 \int_{\mathbb{RP}^n}|\nabla f|^2\,d\sigma.
\]
Letting $\delta\to0$ proves
\eqref{eq:quantile-projective} for $f$.
In particular, $f^\sharp\in H^1_{\mathrm{loc}}(0,1)$
and has a locally absolutely continuous representative.

To establish \eqref{eq:quantile-spherical}, we first verify
that $f^*\in H^1(\mathbb S^n)$.
Applying \eqref{eq:quantile-projective} and
\eqref{eq:quantile-spherical} to $f_\ell$ and using
Lemma~\ref{lem:profile}, we obtain
\[
 \sup_\ell
 \int_{\mathbb S^n}|\nabla f_\ell^*|^2\,d\sigma
 \le
 2^{-2/n}\sup_\ell
 \int_{\mathbb{RP}^n}|\nabla f_\ell|^2\,d\sigma
 <\infty.
\]
Since $f_\ell^*\to f^*$ in $L^2(\mathbb S^n)$,
it follows that $f^*\in H^1(\mathbb S^n)$.
Using \eqref{eq:rearrangement-definition} and substituting
$v=V(\theta)$ in the radial Dirichlet integral now gives
\[
 \int_{\mathbb S^n}|\nabla f^*|^2\,d\sigma
 =
 \int_0^1 I_{\mathbb S^n}(v)^2
             |(f^\sharp)'(v)|^2\,dv,
\]
which proves \eqref{eq:quantile-spherical} for $f$.
Thus both \eqref{eq:quantile-projective} and
\eqref{eq:quantile-spherical} hold for general nonnegative
$f\in H^1(\mathbb{RP}^n)$, and Lemma~\ref{lem:profile}
implies \eqref{eq:dirichlet-transfer}.

Combining \eqref{eq:quantile-projective} and
\eqref{eq:quantile-spherical}, we conclude that
\begin{align*}
 &\int_{\mathbb{RP}^n}|\nabla f|^2\,d\sigma
 -2^{2/n}\int_{\mathbb S^n}|\nabla f^*|^2\,d\sigma\ge
 \int_0^1
 \left[
 I_{\mathbb{RP}^n}(v)^2
 -2^{2/n}I_{\mathbb S^n}(v)^2
 \right]
 |(f^\sharp)'(v)|^2\,dv.
\end{align*}
If $f$ is non-constant, then $f^\sharp$ is non-constant by
equimeasurability. Since $f^\sharp$ is locally absolutely
continuous, $(f^\sharp)'$ is nonzero on a set of positive
measure. Lemma~\ref{lem:profile} therefore implies that the
right-hand side of the preceding inequality is strictly
positive. Thus \eqref{eq:dirichlet-transfer} is strict
unless $f$ is constant, and constant functions give equality.
\end{proof}

Combining Lemma~\ref{lem:rearrangement-energy} with the sharp
Sobolev inequality on $\mathbb S^n$, we obtain Theorem~\ref{thm:main}.

\begin{proposition}\label{prop:local}
Inequality \eqref{eq:main} holds, with equality
if and only if $u$ is constant.
\end{proposition}

\begin{proof}
Let $u\in H^1(\mathbb S^n)$ be real and antipodally even.
Write $u=f\circ\pi$ with $f\in H^1(\mathbb{RP}^n)$.
By \eqref{eq:ordinary} and
Lemma~\ref{lem:rearrangement-energy} applied to $|f|$,
\[
 \|f\|_p^2
 =\|f^*\|_p^2
 \le \|f\|_2^2+\frac{4}{n(n-2)}\|\nabla f^*\|_2^2
 \le \|f\|_2^2
     +\frac{4\,2^{-2/n}}{n(n-2)}\|\nabla f\|_2^2,
\]
where we used $\|\nabla|f|\|_2\le\|\nabla f\|_2$.
Since $\pi$ is a local isometry and the measures are normalized,
$u$ and $f$ have the same $L^r$ norms and Dirichlet energy.
Thus the preceding estimate proves \eqref{eq:main}.

If equality holds, Lemma~\ref{lem:rearrangement-energy}
implies that $|f|$ is constant.
The normalization of the measure then gives
$\|f\|_p=\|f\|_2$, so equality in the displayed estimate
forces $\|\nabla f\|_2=0$.
Since $\mathbb{RP}^n$ is connected, $f$, and hence $u$,
is constant. Conversely, constant functions give equality.
\end{proof}

\begin{remark}\label{rem:sharpness}
The coefficient $2^{-2/n}$ in Proposition~\ref{prop:local}
is optimal. To see this, fix $e\in\mathbb S^n$ and define
\begin{equation}
 U_{\varepsilon,e}(\xi)
 =
 \left(
 \frac{\varepsilon}
 {1-\sqrt{1-\varepsilon^2}\,\xi\cdot e}
 \right)^{(n-2)/2},
 \qquad 0<\varepsilon<1.
\end{equation}
Conformal covariance gives
\begin{equation}\label{eq:bubble-identities}
 \int_{\mathbb S^n}U_{\varepsilon,e}^{p}\,d\sigma=1,
 \qquad
 P_2(U_{\varepsilon,e})
 =U_{\varepsilon,e}^{p-1}.
\end{equation}
As $\varepsilon\to0$, the measures
$U_{\varepsilon,e}^{p}\,d\sigma$ converge weakly to
the Dirac measure at $e$, and
$\|U_{\varepsilon,e}\|_2\to0$.

Set
\[
 u_\varepsilon
 =U_{\varepsilon,e}+U_{\varepsilon,-e},
\]
which is antipodally even.
Applying H\"older's inequality on the two hemispheres
determined by $e$, we obtain
\[
 \int_{\mathbb S^n}
 U_{\varepsilon,e}^{p-1}
 U_{\varepsilon,-e}\,d\sigma
 \to0.
\]
Together with \eqref{eq:bubble-identities} and concentration
at the two antipodal points, this gives
\[
 \int_{\mathbb S^n}u_\varepsilon^{p}\,d\sigma\to2,
 \qquad
 \|u_\varepsilon\|_2^2\to0,
 \qquad
 \langle u_\varepsilon,P_2u_\varepsilon\rangle\to2.
\]
Consequently,
\[
 \frac{\|u_\varepsilon\|_{p}^2-\|u_\varepsilon\|_2^2}
 {\langle u_\varepsilon,P_2u_\varepsilon\rangle
  -\|u_\varepsilon\|_2^2}
 \to
 \frac{2^{2/p}}{2}
 =2^{-2/n}
 =\alpha.
\]
This proves the optimality of $\alpha$ in \eqref{eq:main}.
\end{remark}

\section{The centered axially-symmetric problem}
\label{sec:axial-1}

We now turn to the proof of Theorem~\ref{thm:axial}. Fix an axis and write $
x=\xi_{n+1}$ for the corresponding coordinate function. An axially-symmetric function depends only on \(x\), but need not be even. Since the transverse first moments of an axial density vanish automatically, the scalar condition
\[
\int_{\mathbb S^n}x|u|^p\,d\sigma=0
\]
is equivalent to the full barycenter condition \eqref{balanced}.

For axially-symmetric functions, the normalized spherical measure and the Laplacian reduce to
\[
 d\nu_n(x)
 =
 \frac{\Gamma((n+1)/2)}
 {\sqrt\pi\,\Gamma(n/2)}
 (1-x^2)^{(n-2)/2}\,dx,
 \qquad
 \Delta_n f=(1-x^2)f''-nxf'.
\]
Throughout the axial argument, unless otherwise indicated, all integrals are taken over \([-1,1]\) with respect to \(d\nu_n\). Set
\begin{align}
 E(f)
 &=
 \int_{-1}^1 f^2\,d\nu_n
 +
 \frac{4\alpha}{n(n-2)}
 \int_{-1}^1(1-x^2)(f')^2\,d\nu_n,
 \label{ax:eq:2}\\
 \mathcal S(f,g)
 &=
 \int_{-1}^1 fg\,d\nu_n
 +
 \frac4{n(n-2)}
 \int_{-1}^1(1-x^2)f'g'\,d\nu_n.\notag
\end{align}
The classical sharp Sobolev inequality on \(\mathbb S^n\) becomes
\begin{equation}\label{ax:eq:4}
 \mathcal S(f,f)\ge \|f\|_p^2.
\end{equation}
Improved Sobolev inequalities under moment constraints were studied
in \cite{HangWang}. Here we seek the optimal coefficient of the
Dirichlet term while keeping the coefficient of the $L^2$ term equal
to one.

To formulate the variational problem, we first describe the possible
loss of compactness at the critical exponent. Let $\{u_j\}$ be a bounded
sequence of axially-symmetric functions in $H^1(\mathbb S^n)$, and, after passing
to a subsequence, 
\begin{equation*}
u_j\rightharpoonup u\text{ weakly in }H^1(\mathbb{S}^n).
\end{equation*}
for some $u\in H^1(\mathbb{S}^n)$. 

On every closed interval
$I\Subset(-1,1)$, both weights $(1-x^2)^{(n-2)/2}$ and
$(1-x^2)^{n/2}$ are bounded away from zero. Consequently, there exists a constant $C_I$ depending on $I$ such that
\[
 \int_I\bigl(|u_j|^2+|u_j'|^2\bigr)\,dx
 \le C_I\int_{-1}^1
 \bigl(|u_j|^2+(1-x^2)|u_j'|^2\bigr)\,d\nu_n,
\]
so the profiles $u_j$ are bounded in the ordinary one-dimensional
space $H^1(I)$. The compact embedding of $H^1(I)$ into $C^0(I)$
and a diagonal argument therefore give the local uniform convergence to $u$ on $(-1,1)$. In particular, we have
\[
 |u_j|^p\longrightarrow |u|^p
 \quad\text{in }L^1_{\mathrm{loc}}\bigl((-1,1),d\nu_n\bigr).
\]
Thus any weak limit of the critical mass measures
$|u_j|^p\,d\nu_n$ agrees with $|u|^p\,d\nu_n$ in the interior.
Any defect mass must therefore be supported at the two poles
$x=\pm1$. To be more specific, we have the following lemma.
\begin{lemma}\label{ax:lem:2.1}
Let $\{u_j\}$ be a bounded sequence of nonnegative axially-symmetric functions in
$H^1(\mathbb S^n)$. After passing to a subsequence, there exist a
nonnegative axially-symmetric function $u\in H^1(\mathbb S^n)$ and numbers
$\rho_\pm\ge0$ such that $u_j\rightharpoonup u$ in $H^1$,
$u_j\to u$ in $L^2$, and
\[
 u_j^p\,d\nu_n\rightharpoonup
 u^p\,d\nu_n+\rho_+\delta_1+\rho_-\delta_{-1}.
\]
Moreover,
\[
 \liminf_{j\to\infty}E(u_j)
 \ge E(u)+\alpha\bigl(
 \rho_+^{(n-2)/n}+\rho_-^{(n-2)/n}\bigr).
\]
\end{lemma}

\begin{proof}
Weak compactness and the compact embedding
$H^1(\mathbb S^n)\hookrightarrow L^2(\mathbb S^n)$ give the claimed convergence of $u_j$. The asserted measure convergence follows from the interior
compactness discussed above and the uniform $L^p$ bound supplied
by the classical Sobolev inequality \eqref{ax:eq:4}.

Set $v_j=u_j-u$. Weak convergence gives the energy splitting
\[
 E(u_j)=E(u)+E(v_j)+o(1),
\]
where $o(1)\to 0$ as $j\to \infty$.

For any $0<\delta<\frac{1}{2}$ small, choose smooth cut-off functions $\chi_{\pm,\delta}$ with 
\begin{equation*}
\chi_{\pm,\delta}(x)=
\begin{cases}
1, \qquad |\pm1-x|<\delta,\\
0, \qquad |\pm1-x|>2\delta.
\end{cases}
\end{equation*}
and $0\le\chi_{\pm,\delta}\le1$. Since $v_j\to 0$ in $L^2$ and their uniform boundedness in $H^1$, we have 
\[
 \int(1-x^2)|(\chi_{\pm,\delta} v_j)'|^2\,d\nu_n
 =
 \int(1-x^2)\chi_{\pm,\delta}^2|v_j'|^2\,d\nu_n+o_{\delta}(1).\,
\]
where $o_{\delta}(1)\to 0$ as $j\to \infty$ for every fixed $\delta$. 

Since the supports of $\chi_{\pm,\delta}$ are disjoint, the classical Sobolev inequality \eqref{ax:eq:4} yields
\begin{align*}
 E(v_j)
 &\ge
 \alpha\mathcal S(\chi_{+,\delta}v_j,\chi_{+,\delta}v_j)
 +\alpha\mathcal S(\chi_{-,\delta}v_j,\chi_{-,\delta}v_j)-o_{\delta}(1)\\
 &\ge
 \alpha\bigl(\|\chi_{+,\delta}v_j\|_p^2+\|\chi_{-,\delta}v_j\|_p^2\bigr)-o_{\delta}(1).
\end{align*}
By the measure convergence and the reverse triangle inequality, we have
\[
 \liminf_{j\to\infty}\|\chi_{\pm,\delta} v_j\|_p
 \ge
 \bigl(\rho_\pm+\|\chi_{\pm,\delta} u\|_p^p\bigr)^{1/p}
 -\|\chi_{\pm,\delta} u\|_p.
\]

Consequently, for every fixed $\delta$, we have
\begin{equation*}
\liminf_{j\to \infty} E(u_j)\geq E(u)+\alpha\sum_{\pm} \left[\bigl(\rho_\pm+\|\chi_{\pm,\delta} u\|_p^p\bigr)^{1/p}
 -\|\chi_{\pm,\delta} u\|_p\right]^2.
\end{equation*}

Now we take $\delta \to 0^+$. Since
$u\in L^p$, we have $\|\chi_{\pm,\delta} u\|_p\to0$ as $\delta \to 0^+$. Then the desired bound follows
from the energy splitting and $2/p=(n-2)/n$.
\end{proof}

This motivates a relaxed variational problem in which
the critical mass is allowed to have an atom at each pole. More precisely, for $u\in H^1(\mathbb{S}^n)$ axially-symmetric and \(r_+,r_-\ge0\), we define the relaxed energy by
\[
 \mathcal E_{\mathrm{rel}}(u,r_+,r_-)
 :=
 E(u)
 +
 \alpha\Bigl(
 r_+^{(n-2)/n}+r_-^{(n-2)/n}
 \Bigr).
\]
and impose the constraint
\begin{equation}\label{ax:eq:5}
\mathcal{L}_{\mathrm{rel}}:=\left\{(u,r_+,r_-)\in H^1(\mathbb{S}^n)\times [0,\infty)^2:  \int_{-1}^1|u|^p\,d\nu_n+r_++r_-=1,
 \ 
 \int_{-1}^1x|u|^p\,d\nu_n+r_+-r_-=0\right\}
\end{equation}

In the following, we will establish attainment of the relaxed problem and exclude mixtures of regular and singular mass. Every relaxed minimizer is therefore either a Dirac measure with two antipodal poles or a smooth positive balanced
function. The latter case is treated by the instability argument in the following sections.

We next establish attainment of the relaxed infimum and describe the Euler--Lagrange equation and regularity of its minimizers.

\begin{proposition}\label{ax:prop:2.2}
There exists a minimizer $(u,r_+,r_-)\in \mathcal{L}_{\mathrm{rel}}$ of $\mathcal E_{\mathrm{rel}}$ with $\mathcal E_{\mathrm{rel}}(u,r_+,r_-)\leq 1$. Furthermore,
\begin{enumerate}
    \item  either $(u,r_+,r_-)=(0,\frac{1}{2},\frac{1}{2})$ with $\mathcal{E}_{\mathrm{rel}}(0,\frac{1}{2},\frac{1}{2})=1$
    \item or $u$ is smooth and does not change sign. If we assume $u>0$, then it satisfies 
\begin{equation}\label{ax:eq:6}
 u-\frac{4\alpha}{n(n-2)}\Delta_n u=(a+bx)u^{p-1}
\end{equation}
with $a=\mathcal E_{\mathrm{rel}}(u,r_+,r_-)$ and for some $b\in\mathbb R$ with
\begin{equation}\label{ax:eq:7}
 a\pm b=(2r_{\pm})^{-2/n}\qquad \text{ if }r_\pm>0 \text{ respectively}.
\end{equation}
\end{enumerate}

\end{proposition}

\begin{proof}
Let $\{(u_j,r_{+,j},r_{-,j})\}\subset \mathcal{L}_{\mathrm{rel}}$ be a minimizing sequence. Since replacing $u_j$ by $|u_j|$ preserves $\mathcal{E}_{\mathrm{rel}}$ and $\{(|u_j|,r_{+,j},r_{-,j})\}\subset \mathcal{L}_{\mathrm{rel}}$, we may assume $u_j\geq 0$. After passing
to a subsequence, we assume that $r_{\pm,j}\to \tilde r_\pm$. Applying Lemma~\ref{ax:lem:2.1}, we obtain a weak limit $u$ in $H^1(\mathbb{S}^n)$ and defect masses
$\rho_\pm$.
Write $r_\pm=\tilde r_\pm+\rho_\pm$. Then $(u,r_+,r_-)$
satisfies \eqref{ax:eq:5} since both constraints pass to the limit. The energy bound in Lemma~\ref{ax:lem:2.1}, together with the inequality
\[
 (r+\rho)^{(n-2)/n}
 \le r^{(n-2)/n}+\rho^{(n-2)/n},
 \qquad r,\rho\ge0,
\]
shows that $(u,r_+,r_-)$ attains the relaxed infimum.

If $u\equiv 0$, then the constraints \eqref{ax:eq:5} force $r_\pm=\frac 12$ and then $\mathcal{E}_{\mathrm{rel}}(u,r_+,r_-)=\mathcal{E}_{\mathrm{rel}}(0,\frac{1}{2},\frac{1}{2})=1$.

Now suppose $u\not\equiv0$. Let
\begin{equation*}
G_0(u,r_+,r_-):=\int_{-1}^1|u|^p\,d\nu_n+r_++r_--1
\end{equation*}
and
\begin{equation*}
G_1(u,r_+,r_-):=\int_{-1}^1x|u|^p\,d\nu_n+r_+-r_-
\end{equation*}
be the two constraints. Then we have for any axially-symmetric $h\in H^1(\mathbb{S}^n)$,
\begin{equation*}
D_uG_0[h]=p\int_{-1}^1|u|^{p-2}uh\,d\nu_n
\end{equation*}
and
\begin{equation*}
D_uG_1[h]=p\int_{-1}^1x|u|^{p-2}uh\,d\nu_n.
\end{equation*}

We next verify that $D_uG_0$ and $D_uG_1$ are linearly independent. Indeed, suppose there exists $c,d\in \mathbb R$ such that
\begin{equation*}
cD_uG_0[h]+dD_uG_1[h]=0
\end{equation*}
for any axially-symmetric $h\in H^1(\mathbb{S}^n)$. Taking $h=(c+dx)u$ gives
\begin{equation*}
\int_{-1}^1(c+dx)^2|u|^{p}\,d\nu_n=0,
\end{equation*}
which forces $c=d=0$.

Since $\mathcal{E}_{\mathrm{rel}}$, $G_0$ and $G_1$ are $C^2$ functions of $(u,r_+,r_-)$, there exist Lagrange multipliers $a,b\in\mathbb R$ such that the Lagrangian
\begin{align*}
 L(v,s_+,s_-)
 &:=
 E(v)+\alpha\bigl(s_+^{2/p}+s_-^{2/p}\bigr)-
 \frac{2a}{p}
 \left(\int_{-1}^1|v|^p\,d\nu_n+s_++s_--1\right)\\
 &\quad-
 \frac{2b}{p}
 \left(\int_{-1}^1x|v|^p\,d\nu_n+s_+-s_-\right)
\end{align*}
is stationary at $(u,r_+,r_-)$.

Differentiating $L$ in any axially-symmetric direction
$h\in H^1(\mathbb S^n)$ gives
\begin{equation*}
\int_{-1}^1uh\,d\nu_n
 +\frac{4\alpha}{n(n-2)}
 \int_{-1}^1(1-x^2)u'h'\,d\nu_n
 =
 \int_{-1}^1(a+bx)|u|^{p-2}uh\,d\nu_n,    
\end{equation*}
which is the weak form of \eqref{ax:eq:6}.

Next, stationarity in $s_{\pm}$ gives
\[
 0=
 \left.\frac{\partial\mathcal L}{\partial s_\pm}
 \right|_{(u,r_+,r_-)}
 =
 \frac{2\alpha}{p}r_{\pm}^{2/p-1}
 -\frac{2}{p}(a\pm b).
\]
Since $2/p-1=-2/n$ and $\alpha=2^{-2/n}$, we obtain
\[
 a\pm b=\alpha r_{\pm}^{-2/n}=(2r_{\pm})^{-2/n},
\]
as desired in \eqref{ax:eq:7}.

Testing \eqref{ax:eq:6} against $u$, adding \eqref{ax:eq:7} multiplied by $r_{\pm}$ respectively, and using $u\in \mathcal{L}_{\mathrm{rel}}$, we obtain
\begin{align*}
 E(u)+\alpha\bigl(r_+^{(n-2)/n}+r_-^{(n-2)/n}\bigr)=
 a\left(\int_{-1}^1|u|^p\,d\nu_n+r_++r_-\right)+
 b\left(\int_{-1}^1x|u|^p\,d\nu_n+r_+-r_-\right)
 =a.
\end{align*}

\vspace{3mm}

We claim that \eqref{ax:eq:6} actually holds weakly on the whole sphere. Indeed, for any $h\in H^1(\mathbb{S}^n)$, we define
\begin{equation*}
\bar h(\xi)=\int_{SO(n)}h(R\xi)dR,
\end{equation*}
where $R\in SO(n)$ act on the $(\xi_1,\dots, \xi_{n})$ and $dR$ is the normalized Haar measure. Then $\bar h$ is axially-symmetric. Since $u$ is axially-symmetric, we have
\begin{equation*}
\int_{\mathbb{S}^n}\nabla u\cdot \nabla (h\circ R)d\sigma =\int_{\mathbb{S}^n}\nabla (u\circ R^{-1})\cdot \nabla hd\sigma=\int_{\mathbb{S}^n}\nabla u\cdot \nabla hd\sigma.
\end{equation*}
Averaging over $R$ gives
\begin{equation*}
\int_{\mathbb{S}^n}\nabla u\cdot \nabla \bar hd\sigma=\int_{\mathbb{S}^n}\nabla u\cdot \nabla hd\sigma.
\end{equation*}

Since $a+bx$ is also axially-symmetric, we can compute similarly to show that
\begin{equation*}
\int_{\mathbb{S}^n}uh\,d\sigma
 +\frac{4\alpha}{n(n-2)}
 \int_{\mathbb{S}^n}\nabla u\cdot \nabla h\,d\sigma
 =
 \int_{\mathbb{S}^n}(a+bx)|u|^{p-2}uh\,d\sigma    
\end{equation*}
as claimed.

It remains to prove the regularity and positivity of $u$. Write the weak
equation on $\mathbb S^n$ as $-\Delta_{\mathbb S^n}u=Vu$, where
\[
 V=\frac{n(n-2)}{4\alpha}
 \bigl((a+bx)|u|^{\frac{4}{n-2}}u-1\bigr)
 \in L^{n/2}(\mathbb S^n).
\]

For $r>1$ and $L>0$, we test the equation against $u\min(u,L)^{2r-2}$ and set
$w_L=u\min(u,L)^{r-1}$. The chain rule gives
\[
 \int_{\mathbb S^n}|\nabla w_L|^2\,d\sigma
 \le r^2\int_{\mathbb S^n}|V|w_L^2\,d\sigma.
\]

Then by the Sobolev inequality \eqref{ax:eq:4}, we have for any $M_r>0$,
\begin{equation*}
\begin{aligned}
\|w_L\|_p^2
&\leq \|w_L\|_2^2+\frac{4r^2}{n(n-2)}\int_{\mathbb S^n}|V|w_L^2\,d\sigma\\
&\leq \left(1+\frac{4r^2M_r}{n(n-2)}\right) \|w_L\|_2^2+\frac{4r^2}{n(n-2)}
 \int_{\{|V|>M_r\}}|V|w_L^2\,d\sigma\\
&\leq \left(1+\frac{4r^2M_r}{n(n-2)}\right) \|w_L\|_2^2+\frac{4r^2}{n(n-2)}
 \bigl\|V\mathbf1_{\{|V|>M_r\}}\bigr\|_{n/2}\|w_L\|_p^2.
\end{aligned}
\end{equation*}

Choose $M_r>0$ sufficiently large that
\[
 \frac{4r^2}{n(n-2)}
 \bigl\|V\mathbf1_{\{|V|>M_r\}}\bigr\|_{n/2}
 \le\frac12.
\]
Then we have
\[
 \|w_L\|_p^2
 \le
 2\left(1+\frac{4r^2M_r}{n(n-2)}\right)\|w_L\|_2^2\leq 2\left(1+\frac{4r^2M_r}{n(n-2)}\right)\int_{\mathbb{S}^n}|u|^{2r}\,d\sigma .
\]

If $u\in L^{2r}$, taking $L\to\infty$ gives
$u\in L^{pr}$. Starting from $u\in L^p$ and iterating this
implication proves that $u$ belongs to every finite $L^q$ space.
Elliptic estimates and Sobolev embedding then give boundedness
and H\"older regularity; Schauder estimates yield
$u\in C^{2,\theta}(\mathbb S^n)$ for some $\theta\in(0,1)$.

In particular, $V$ is bounded. Choose $K\ge\|V_-\|_\infty$.
Since
\[
 (-\Delta_{\mathbb S^n}+K)u=(V+K)u\ge0,
\]
the strong maximum principle and $u\not\equiv0$ imply that $u$ does not change sign. We may assume $u>0$. Then
its positive minimum on the compact sphere makes the nonlinearity
smooth along $u$, so elliptic bootstrapping gives
$u\in C^\infty(\mathbb S^n)$.
\end{proof}

We now complete the compactness reduction by excluding mixtures of regular and singular mass.

\begin{proposition}\label{ax:prop:3.1}
If $(u,r_+,r_-)\in \mathcal{L}_{\mathrm{rel}}$ is a minimizer of $\mathcal{E}_{\mathrm{rel}}$ with $u\not\equiv 0$. Then $r_+=r_-=0$.
\end{proposition}

\begin{proof}
Let $(u,r_+,r_-)\in \mathcal{L}_{\mathrm{rel}}$ be a minimizer of $\mathcal{E}_{\mathrm{rel}}$ with $u\not\equiv 0$.
By Proposition~\ref{ax:prop:2.2}, $u$ is smooth
and strictly positive, and $a=\mathcal{E}_{\mathrm{rel}}(u,r_+,r_-)$. Hence, $0<a\le1$. Write
$d\mu=u^p\,d\nu_n$ for the regular critical mass.

Suppose on the contrary that $r_+>0$. Then \eqref{ax:eq:5} gives
\[
 2r_+=1-\int_{-1}^1(1+x)\,d\mu<1.
\]
Thus \eqref{ax:eq:7} implies $a+b>1$. If in addition $r_->0$, similarly, we have $2r_-<1$. Adding these inequalities gives $a>1$, a contradiction. Hence we have $r_-=0$ and then $b>0$. 

For simplicity, we write $r=r_+$. The constraints become
\[
 \int_{-1}^1d\mu=1-r,
 \qquad
 \int_{-1}^1x\,d\mu=-r.
\]

Write
\[
 m=\int_{-1}^1x^2\,d\mu+r.
\]
Since
\[
 1-m=\int_{-1}^1(1-x^2)\,d\mu>0,
\]
we have $m<1$. 

On the other hand, Cauchy--Schwarz inequality gives
\[
 r^2
 =\left(\int_{-1}^1x\,d\mu\right)^2
 \le(1-r)\int_{-1}^1x^2\,d\mu
 =(1-r)(m-r).
\]
Consequently,
\[
 \frac{r}{1-r}\le m<1.
\]

Testing \eqref{ax:eq:6} against
$(1-x^2)u'-\frac{n-2}{2}xu$ and integrating by parts yield the Kazdan-Warner condition
\begin{equation}\label{ax:eq:KW}
 (1-\alpha)\int_{-1}^1xu^2\,d\nu_n
 =-\frac{n-2}{2n}b(1-m).    
\end{equation}

Since $\int_{-1}^1(1-x)\,d\mu=1$, applying H\"older's inequality and Sobolev inequality to $\sqrt{1-x}\,u$ gives
\begin{equation}\label{ax:eq:Ssqrt1-xu}
1\le
 \left(\int_{-1}^1(1-x)^{p/2}\,d\mu\right)^{2/p}
 (1-r)^{2/n}
 =
 \|\sqrt{1-x}\,u\|_p^2(1-r)^{2/n}\leq \mathcal S(\sqrt{1-x}\,u,\sqrt{1-x}\,u)(1-r)^{2/n}.  
\end{equation}

Next, we test \eqref{ax:eq:6} against $(1-x)u$ to get
\begin{equation*}
\begin{aligned}
&\int_{-1}^1(1-x)u^2\,d\nu_n+\frac{4\alpha}{n(n-2)}\int_{-1}^1[(1-x^2)(1-x)(u')^2-(1-x^2)uu']\,d\nu_n\\
=&\int_{-1}^1(a+bx)(1-x)d\mu.
\end{aligned}
\end{equation*}
Plugging
\begin{equation*}
(1-x^2)[(\sqrt{1-x}u)']^2=(1-x^2)(1-x)(u')^2-(1-x^2)uu'+\frac{1+x}{4}u^2,
\end{equation*}
into the identity above, we have
\begin{align*}
 \alpha\mathcal S(\sqrt{1-x}\,u,\sqrt{1-x}\,u)
 &=
 \int_{-1}^1(a+bx)(1-x)\,d\mu\\
 &\quad-
 (1-\alpha)\int_{-1}^1(1-x)u^2\,d\nu_n
 +\frac{\alpha}{n(n-2)}
 \int_{-1}^1(1+x)u^2\,d\nu_n.
\end{align*}

Combining it with \eqref{ax:eq:KW} and \eqref{ax:eq:Ssqrt1-xu}, we obtain
\begin{equation}\label{ax:eq:8}
 \begin{aligned}
 [2(1-r)]^{-2/n}
 &\le
 \alpha\mathcal S(\sqrt{1-x}\,u,\sqrt{1-x}\,u)=a-b\bigl[m+\theta_n(1-m)\bigr]
 -\gamma_n\int_{-1}^1u^2\,d\nu_n,
 \end{aligned}
\end{equation}
where
\[
 \gamma_n:=1-\alpha-\frac{\alpha}{n(n-2)},
 \qquad
 \theta_n:=\frac{n-2}{2n}
 \left(1+\frac{\alpha}{n(n-2)(1-\alpha)}\right).
\]
The elementary inequality $e^s>1+s$ for $s>0$ gives
\[
 n(n-2)(\alpha^{-1}-1)
 >2(n-2)\log2>1.
\]
It follows that
\[
 \gamma_n>0,
 \qquad
 \frac{n-2}{2n}<\theta_n<\frac{n-2}{n}<1.
\]

By \eqref{ax:eq:7}, $a+b=(2r)^{-2/n}$. Substituting this into
\eqref{ax:eq:8} gives
\begin{equation*}
a>
 \frac{
 [2(1-r)]^{-2/n}
 +(2r)^{-2/n}\bigl[m+\theta_n(1-m)\bigr]
 }{1+m+\theta_n(1-m)},
\end{equation*}
since
$\gamma_n>0$ and $u\not\equiv0$. Note that $\theta_n<1$  and 
\[
 (2r)^{-2/n}>[2(1-r)]^{-2/n}.
\]
Then the lower bound of $a$ is
increasing in $m$.

Thus, using $m\ge r/(1-r)$ and setting $t=1-2r\in(0,1)$, we find
\begin{equation*}
 \begin{aligned}
 a>
 \frac{
 (1+t)^{(n-2)/n}+(1-t)^{(n-2)/n}
 +2\theta_nt(1-t)^{-2/n}
 }{2(1+\theta_nt)}
 \ge1,
 \end{aligned}
\end{equation*}
contradicting $a\le1$. Consequently, we get $r_+=0$. Similarly, we can show that $r_-=0$.
\end{proof}

\section{Variational identities and moment estimates}
\label{sec:axial-4}

Having excluded mixed minimizers in Proposition~\ref{ax:prop:3.1}, we now turn to smooth positive solutions of the Euler--Lagrange equation. We derive identities for the constrained second variation and establish the moment estimates needed for the instability argument in Section~\ref{sec:axial-instability}.

\subsection{The constrained conformal variation}\label{sec:axial-variation}

Let $u$ be a smooth positive axially-symmetric solution of the normalized Euler--Lagrange equation
\begin{equation}\label{ax:eq:10}
\begin{gathered}
u-\frac{4\alpha}{n(n-2)}\Delta_nu=(a+bx)u^{p-1}
\end{gathered}
\end{equation}
with the constraint
\begin{equation*}
u\in \mathcal{L}_r:=\left\{u\in H^1(\mathbb{S}^n):\ u=u(x)>0,\  \int_{-1}^1u^p\,d\nu_n=1,
 \ 
 \int_{-1}^1xu^p\,d\nu_n=0\right\}
\end{equation*}

Let $d\mu=u^p\,d\nu_n$ be a probability measure, and
$m=\int_{-1}^1x^2\,d\mu\in(0,1)$. According to Proposition \ref{ax:prop:2.2} and Proposition \ref{ax:prop:3.1}, $a=E(u)\in (0,1]$.
The quadratic form associated to \eqref{ax:eq:10} is given by
\begin{equation*}
\begin{aligned}
Q(h,k)={}&\int_{-1}^1hk\,d\nu_n
+\frac{4\alpha}{n(n-2)}\int_{-1}^1(1-x^2)h'k'\,d\nu_n\\
&-(p-1)\int_{-1}^1(a+bx)u^{p-2}hk\,d\nu_n.
\end{aligned}
\end{equation*}
A smooth minimizer is stable under the two constraints if $Q(h,h)\geq0$ for every tangent direction 
\begin{equation*}
T_u\mathcal{L}_r=\left\{h\in H^1(\mathbb{S}^n):\ h=h(x),\  \int_{-1}^1u^{p-1}h\,d\nu_n=0,
 \ 
 \int_{-1}^1xu^{p-1}h\,d\nu_n=0\right\}.
\end{equation*}

The next lemma constructs a corrected conformal variation tangent to both constraints and gives an explicit formula for the corresponding second variation.
\begin{lemma}
Let
\begin{equation}\label{ax:eq:12}
t=\frac{(n-2)(1-m)}{2nm},\qquad
\phi=(1-x^2)u'-\left(\frac{n-2}{2}-t\right)xu.
\end{equation}
Then we have $\phi\in T_u \mathcal{L}_r$ and
\begin{equation}\label{ax:eq:15}
\begin{aligned}
Q(\phi,\phi)={}&-\frac{8\alpha t}{n(n-2)}\int_{-1}^1\phi^2\,d\nu_n\\
&+\left(1-\alpha+\frac{4\alpha t(1+t)}{n(n-2)}\right)
\left[(3+2t)\int_{-1}^1x^2u^2\,d\nu_n-\int_{-1}^1u^2\,d\nu_n\right]\\
&-\left(1+\frac{4t}{n-2}\right)\left(\frac{n-2}{n}+t\right)
b\int_{-1}^1x^3\,d\mu.
\end{aligned}
\end{equation}   
\end{lemma}
\begin{proof}
 For any smooth function $F$ on $(-1,1)$, integration by parts gives
\[
\int_{-1}^1(1-x^2)F'\,d\nu_n=n\int_{-1}^1xF\,d\nu_n,
\]
\begin{equation}\label{ax:eq:13}
\int_{-1}^1Fu^{p-1}\phi\,d\nu_n
=-\frac{n-2}{2n}\int_{-1}^1(1-x^2)F'\,d\mu
+t\int_{-1}^1xF\,d\mu.
\end{equation}

Taking $F=1$ and $x$ respectively, we get
\begin{equation*}
\int_{-1}^1u^{p-1}\phi\,d\nu_n
=\int_{-1}^1xu^{p-1}\phi\,d\nu_n=0
\end{equation*}
Thus $\phi$ belongs to $T_u \mathcal{L}_r$. 

Let $\mathbb L=1-\frac{4\alpha}{n(n-2)}\Delta_n
-(p-1)(a+bx)u^{p-2}$ be the linearized operator. Direct computation gives
\begin{equation}\label{ax:eq:14}
\begin{aligned}
\mathbb L(\phi)
={}&-\frac{8\alpha t}{n(n-2)}\phi+2\left(1-\alpha+\frac{4\alpha t(1+t)}{n(n-2)}\right)xu+\left[b-\frac{4ta}{n-2}x-b\left(1+\frac{4t}{n-2}\right)x^2\right]
u^{p-1}.
\end{aligned}
\end{equation}

Since $\phi\in T_u\mathcal{L}_r$, test \eqref{ax:eq:14} against $\phi$ gives
\begin{align*}
Q(\phi,\phi)=&-\frac{8\alpha t}{n(n-2)}\int_{-1}^{1}\phi^2\,d\nu_n+2\left(1-\alpha+\frac{4\alpha t(1+t)}{n(n-2)}\right)\int_{-1}^{1}xu\phi\,d\nu_n\\
&-b\left(1+\frac{4t}{n-2}\right)\int_{-1}^{1}x^2
u^{p-1}\phi\,d\nu_n.
\end{align*}

Substituting the definition of $\phi$ in \eqref{ax:eq:12} and
integrating by parts, we obtain
\begin{align*}
 \int_{-1}^1xu\phi\,d\nu_n
 &=
 \frac12\int_{-1}^1x(1-x^2)(u^2)'\,d\nu_n
 -\left(\frac{n-2}{2}-t\right)
 \int_{-1}^1x^2u^2\,d\nu_n\\
 &=
 \frac{3+2t}{2}\int_{-1}^1x^2u^2\,d\nu_n
 -\frac12\int_{-1}^1u^2\,d\nu_n.
\end{align*}

For the third integral, taking $F=x^2$ in \eqref{ax:eq:13} and using $\int_{-1}^1x\,d\mu=0$
gives
\begin{align*}
 \int_{-1}^1x^2u^{p-1}\phi\,d\nu_n
 &=
 -\frac{n-2}{n}\int_{-1}^1x(1-x^2)\,d\mu
 +t\int_{-1}^1x^3\,d\mu\\
 &=
 \left(\frac{n-2}{n}+t\right)
 \int_{-1}^1x^3\,d\mu.
\end{align*}

Combining them, we obtain \eqref{ax:eq:15} as desired.
\end{proof}

Since $t>0$, the first term in the last lemma is nonpositive and vanishes only if $\phi\equiv0$. In that case, evaluating the definition of $\phi$ in \eqref{ax:eq:12} at $x=1$ and using $u(1)>0$ gives $t=(n-2)/2$, and hence
$(1-x^2)u'=0$. Thus the first term is strictly negative whenever $u$ is non-constant. The remaining two terms, however, have no sign determined by the centering condition: the second involves the second moment of the $L^2$ density, while the third couples the
multiplier $b$ to the third moment of the critical density.
To prove instability, we must control their combined contribution so that it cannot offset the negative square term. For this purpose, we use the Euler--Lagrange equation to derive further identities
and rewrite the second variation in forms suited to the moment estimates and Sobolev inequalities below.

\subsection{Integral identities and moment bounds}

We next record the identities obtained by testing the Euler--Lagrange
equation.

\begin{lemma}
Let $u\in \mathcal{L}_r$ be a solution to \eqref{ax:eq:10}. Then we have
\begin{equation}\label{ax:eq:19}
\frac{4\alpha}{n(n-2)}\int_{-1}^1x(1-x^2)(u')^2\,d\nu_n=b\left[m+\left(\frac1{2(1-\alpha)}-\frac1n\right)(1-m)\right]
\end{equation}
and
\begin{equation}\label{ax:eq:22}
\begin{aligned}
b\int_{-1}^1x^3\,d\mu
={}&\lambda_n\left[\frac{4\alpha}{n-2}\int_{-1}^1(1-x^2)(u')^2\,d\nu_n
-a((n+1)m-1)\right]\\
&+\lambda_n\frac{2n(k(n-2)+3)}{(n-2)(n+2-k)}\frac{4\alpha}{n(n-2)}
\int_{-1}^1(1-x^2)^2(u')^2\,d\nu_n,
\end{aligned}
\end{equation}
where 
\begin{equation}\label{ax:eq:17}
k=n(\alpha^{-1}-1),\qquad
\lambda_n=\frac{2(n+2-k)}{3n^2+8n+8-k(n+4)}.
\end{equation}
\end{lemma}
\begin{proof}

Testing \eqref{ax:eq:10} against $xu$ and integrating by parts gives
\begin{equation*}
bm=\left(1+\frac{2\alpha}{n-2}\right)\int_{-1}^1xu^2\,d\nu_n+\frac{4\alpha}{n(n-2)}\int_{-1}^1x(1-x^2)(u')^2\,d\nu_n.
\end{equation*}
Combining it with the Kazdan-Warner condition \eqref{ax:eq:KW} gives \eqref{ax:eq:19}.

Recall that Proposition \ref{ax:prop:2.2} and Proposition \ref{ax:prop:3.1} yields
\begin{equation}\label{ax:eq:a}
a=E(u)=\int_{-1}^1u^2\,d\nu_n+\frac{4\alpha}{n(n-2)}\int_{-1}^1(1-x^2)(u')^2\,d\nu_n.
\end{equation}

Next we test \eqref{ax:eq:10} against $(1-x^2)u$ and integrate by parts to get
\begin{equation}\label{ax:eq:20}
\begin{aligned}
&\frac{4\alpha}{n(n-2)}\int_{-1}^1(1-x^2)^2(u')^2\,d\nu_n\\
=&a(1-m)-b\int_{-1}^1x^3\,d\mu-\left(1+\frac{4\alpha}{n(n-2)}\right)\int_{-1}^1u^2\,d\nu_n+\left(1+\frac{4\alpha(n+1)}{n(n-2)}\right)\int_{-1}^1x^2u^2\,d\nu_n\\
=&a(1-m)-b\int_{-1}^1x^3\,d\mu-\frac{n}{n+1}\int_{-1}^1u^2\,d\nu_n-\frac{1}{n+1}\left(1+\frac{4\alpha(n+1)}{n(n-2)}\right)\int_{-1}^1(1-(n+1)x^2)u^2\,d\nu_n.
\end{aligned}
\end{equation}

By \eqref{ax:eq:a}, it follows that
\begin{equation}\label{ax:eq:412}
\begin{aligned}
&\left(1+\frac{4\alpha(n+1)}{n(n-2)}\right)\int_{-1}^1(1-(n+1)x^2)u^2\,d\nu_n\\
=&\frac{4\alpha}{n-2}\int_{-1}^1(1-x^2)(u')^2\,d\nu_n-\frac{4\alpha (n+1)}{n(n-2)}\int_{-1}^1(1-x^2)^2(u')^2\,d\nu_n\\
&-a((n+1)m-1)-(n+1)b\int_{-1}^1x^3\,d\mu.
\end{aligned}
\end{equation}

Then we test \eqref{ax:eq:10} against $x(1-x^2)u'$ and integrate by parts to get
\begin{equation}\label{ax:eq:414}
\begin{aligned}
&\frac{4\alpha}{n-2}\int_{-1}^1(1-x^2)(u')^2\,d\nu_n-a((n+1)m-1)\\
=&\frac{n}{n-2}\int_{-1}^1(1-(n+1)x^2)u^2\,d\nu_n+(n+2)b\int_{-1}^1x^3\,d\mu+\frac{4\alpha(n-3)}{(n-2)^2}\int_{-1}^1(1-x^2)^2(u')^2\,d\nu_n.
\end{aligned}
\end{equation}

Combining \eqref{ax:eq:412} and \eqref{ax:eq:414} gives \eqref{ax:eq:22}.

\end{proof}

For a non-constant solution $u$, we define the two auxiliary ratios
\begin{equation}\label{ax:eq:23}
v=\frac{\int_{-1}^1(1-x^2)^2(u')^2\,d\nu_n}
{\int_{-1}^1(1-x^2)(u')^2\,d\nu_n},\qquad
w=\frac{b\int_{-1}^1x^3\,d\mu}
{\frac{4\alpha}{n(n-2)}\int_{-1}^1(1-x^2)(u')^2\,d\nu_n}.
\end{equation}

\begin{lemma}\label{ax:lem:4.2}
If $u$ is a non-constant solution to \eqref{ax:eq:10}, then we have
\begin{equation}\label{ax:eq:24}
 \left|\int_{-1}^1x^3\,d\mu\right|
 \le m(1-m).
\end{equation}
and
\[
 m>\frac1{n+1},\qquad
 0<t<\frac{n-2}{2},\qquad
 0<v<1,\qquad
 |w|<2(1-\alpha)m\sqrt{1-v}.
\]
\end{lemma}

\begin{proof}
Since $d\mu$ is a centered probability measure with second moment
$m$, we have
\[
 0\le
 \int_{-1}^1(1\mp x)(x\pm m)^2\,d\mu
 =
 m(1-m)\mp\int_{-1}^1x^3\,d\mu.
\]
Taking both signs yields \eqref{ax:eq:24}.

Set
\[
 \Gamma
 =
 m+\left(\frac1{2(1-\alpha)}-\frac1n\right)(1-m).
\]
Since $\alpha\geq 2^{-2/3}$ and $0<m<1$, we have $\Gamma>1$.

Since $u$ is smooth and non-constant, the definition of $v$ in \eqref{ax:eq:23} gives $0<v<1$. 
Applying Cauchy--Schwarz inequality to \eqref{ax:eq:19} and using the definition of $v$ in \eqref{ax:eq:23}, we obtain
\begin{align*}
 |b|\Gamma=
 \frac{4\alpha}{n(n-2)}
 \left|\int_{-1}^1x(1-x^2)(u')^2\,d\nu_n\right|&<
 \frac{4\alpha}{n(n-2)}
 \left(\int_{-1}^1(1-x^2)(u')^2\,d\nu_n\right)
 \sqrt{1-v}\\
 &=
 \left(a-\int_{-1}^1u^2\,d\nu_n\right)\sqrt{1-v}.
\end{align*}
Here the last equality follows from $a=E(u)$ and \eqref{ax:eq:2}. 

The definition of $w$ in \eqref{ax:eq:23} and \eqref{ax:eq:24} give
\begin{equation}\label{ax:eq:west}
 |w|
 \le
 \frac{|b|m(1-m)}{a-\int_{-1}^1u^2\,d\nu_n}
 <
 \frac{m(1-m)}{\Gamma}\sqrt{1-v}
 <\frac14.
\end{equation}

We next prove the lower bound for $m$. Recall that
$k=n(2^{2/n}-1)\in (0,2)$ and from \eqref{ax:eq:17} that 
\begin{equation*}
\lambda_n=\frac{2(n+2-k)}{3n^2+8n+8-k(n+4)}.
\end{equation*}
The denominator in \eqref{ax:eq:17}
is therefore positive, and
\[
 n\lambda_n-\frac25
 =
 \frac{4\bigl[n^2+n-4-2k(n-1)\bigr]}
 {5\bigl[3n^2+8n+8-k(n+4)\bigr]}
 >0,
\]

If $m\le1/(n+1)$, then \eqref{ax:eq:22}  gives
\[
 w
 \ge
 n\lambda_n
 +\lambda_n
 \frac{2n\bigl[k(n-2)+3\bigr]}
 {(n-2)(n+2-k)}\,v
 >n\lambda_n>\frac25,
\]
contradicting $|w|<1/4$. Thus $m>1/(n+1)$, and the definition
of $t$ in \eqref{ax:eq:12} yields $0<t<(n-2)/2$.

Finally, the improved lower bound for $m$ implies
\[
 \Gamma-\frac{1-m}{2(1-\alpha)}
 =
 \frac{(n+1)m-1}{n}>0.
\]
Substituting this into \eqref{ax:eq:west}, we obtain
\[
 |w|
 <
 \frac{m(1-m)}{\Gamma}\sqrt{1-v}
 <
 2(1-\alpha)m\sqrt{1-v},
\]
which completes the proof.
\end{proof}

\section{Instability of non-constant axially-symmetric solutions}
\label{sec:axial-instability}

In this section we complete the rigidity argument and prove Theorem \ref{thm:axial} by showing that $\phi\in T_u\mathcal{L}_r$ constructed in Section~\ref{sec:axial-variation}
has strictly negative second variation at every non-constant smooth
solution of \eqref{ax:eq:10}.

\begin{theorem}\label{ax:thm:4.1}
Let $u$ be a solution of \eqref{ax:eq:10} with $a=E(u)\leq 1$. If $u$ is non-constant, then
\[
 Q(\phi,\phi)<0.
\]
Consequently, every constrained stable solution is constant, and hence
\[
 u\equiv1,\qquad a=1,\qquad b=0.
\]
\end{theorem}

Throughout the proof, we assume that $u$ is non-constant, so that
Lemma~\ref{ax:lem:4.2} applies. We first establish some preliminary estimates on $Q(\phi,\phi)$. Then we treat a second-moment range that is common to all dimensions. The remaining range is then handled
separately for $n\ge6$, $n=4,5$ and $n=3$, respectively.

\subsection{Preliminary estimates on $Q(\phi,\phi)$}
In this subsection, we establish preliminary estimates on $Q(\phi,\phi)$. First, we rewrite the second variation $Q(\phi,\phi)$ in useful two forms.
\begin{lemma}\label{ax:lem:5.2}
Under the hypothesis and notations of Lemma \ref{ax:lem:4.2}, we have
\begin{equation}\label{ax:eq:25}
 -Q(\phi,\phi)
 =
 a(1-m)A_n(t)
 -D_n(t)\int_{-1}^1u^2\,d\nu_n
 +K_n(t)b\int_{-1}^1x^3\,d\mu
\end{equation}
and
\begin{equation}\label{ax:eq:27}
 Q(\phi,\phi)
 =
 (-f_n(t)+g_n(t)v+h_n(t)w)\frac{4\alpha}{n(n-2)}
 \int_{-1}^1(1-x^2)(u')^2\,d\nu_n,
\end{equation}
where
\begin{equation}\label{ax:eq:canonical-coefficients}
 \begin{aligned}
 A_n(t)
 &=
 \frac{2t}{n}
 +\frac{(n-2)(n-1)}
 {2\alpha\bigl(k(n-2)+3\bigr)}
 \left[
 (1-\alpha)(3+4t)-\frac{4\alpha t^2}{n(n-2)}
 \right],\\[1ex]
 D_n(t)
 &=
 \left[
 \frac12+\frac{(n-1)(n-3)}
 {2\bigl(k(n-2)+3\bigr)}
 \right]
 \left[
 (1-\alpha)(3+4t)-\frac{4\alpha t^2}{n(n-2)}
 \right]-(1-\alpha)+\frac{4\alpha t^2}{n(n-2)},\\[1ex]
 K_n(t)
 &=
 \frac{n-2}{n}+2t+\frac{4t^2}{n-2}-
 \frac{(n-2)(3n-4)}
 {4\alpha\bigl(k(n-2)+3\bigr)}
 \left[
 (1-\alpha)(3+4t)-\frac{4\alpha t^2}{n(n-2)}
 \right].
 \end{aligned}
\end{equation}
and
\begin{equation}\label{ax:eq:gradient-coefficients}
 \begin{aligned}
 f_n(t)
 &=
 \frac{4\alpha}{n-2}t^2+2(1-\alpha)t,\\[1ex]
 g_n(t)
 &=
 \frac{k\bigl(3(n-2)+k\bigr)}
 {(n+2-k)(n+k)}
 +\frac{4k(n-1)(n-3)}
 {(n-2)(n+2-k)(n+k)}\,t-
 \frac{4\bigl[n(n+2)+k(n-2)(2n-1)\bigr]}
 {(n-2)^2(n+2-k)(n+k)}\,t^2,\\[1ex]
 h_n(t)
 &=
 \frac{2\bigl[n(n+2)-k(n+6)+2k^2\bigr]}
 {(n-2)(n+k)(n+2-k)}\,t^2+
 \left[
 -2+
 \frac{k\bigl[n^2+8n-8-k(3n-4)\bigr]}
 {n(n+k)(n+2-k)}
 \right]t\\
 &\quad-
 \frac{(n-2)\bigl[n(n+2)-k(2-k)\bigr]}
 {n(n+k)(n+2-k)}.
 \end{aligned}
\end{equation}
Furthermore, we have $A_n(t),D_n(t)>0$,
\begin{equation}\label{ax:eq:26}
 -D_n(t)<K_n(t)<2D_n(t) \qquad   \text{ if }0<t\le\frac{n-2}{4},
\end{equation}
\begin{equation}\label{ax:eq:fngn}
f_n(t)>2(1-\alpha)(n-2)/n,\qquad f_n(t)-g_n(t)>(1-\alpha)(n-2)/n,\qquad \text{ if }t\geq\frac{n-2}{4}.
\end{equation}
and
\begin{equation}\label{ax:eq:hn}
0<-h_n(t)<\frac{n-2}{n}+2t.
\end{equation}
\end{lemma}
\begin{proof}
 First, eliminating the weighted Dirichlet integral terms
in \eqref{ax:eq:22} and \eqref{ax:eq:20} gives
\begin{equation}\label{ax:eq:x2u2}
\begin{aligned}
\int_{-1}^1x^2u^2\,d\nu_n
 &=
 \left[
 \frac12+\frac{(n-1)(n-3)}{2\bigl(k(n-2)+3\bigr)}
 \right]\int_{-1}^1u^2\,d\nu_n-
 \frac{(n-2)(n-1)}{2\alpha\bigl(k(n-2)+3\bigr)}
 a(1-m)\\
 &\quad+
 \frac{(n-2)(3n-4)}{4\alpha\bigl(k(n-2)+3\bigr)}
 b\int_{-1}^1x^3\,d\mu.        
\end{aligned}
\end{equation}
Substituting it into \eqref{ax:eq:15} and using
\[
 \frac{4mt^2}{n-2}=\frac{2t}{n}(1-m),
\]
we obtain \eqref{ax:eq:25}. For the bounds of $A_n$, $D_n$ and $K_n$, see Lemma~\ref{ax:lem:A1}. 

Combining \eqref{ax:eq:22}, \eqref{ax:eq:25} and $a=E(u)$, we obtain \eqref{ax:eq:27}. For the estimates of $f_n$, $g_n$ and $h_n$ corresponding to \eqref{ax:eq:fngn} and \eqref{ax:eq:hn}, see \eqref{ax:eq:fngnn4} and \eqref{ax:eq:h-bound} in Lemma~\ref{ax:lem:A4}.
\end{proof}

As a corollary, we obtain an improved lower bound on $m$ if $u$ is a non-constant stable solution.
\begin{corollary}\label{ax:cor:mlower}
If $u\in \mathcal{L}_r$ is a non-constant stable solution to \eqref{ax:eq:10}. Then $m>\frac{2}{n+2}$.
\end{corollary}
\begin{proof}
Suppose on the contrary that
\[
 \frac1{n+1}<m\le\frac2{n+2}.
\]
Since 
\[
 t=\frac{(n-2)(1-m)}{2nm},
\]
this is equivalent to
\[
 \frac{n-2}{4}\le t<\frac{n-2}{2}.
\]

Combining Lemma~\ref{ax:lem:4.2}, \eqref{ax:eq:27}, and \eqref{ax:eq:h-bound} in Lemma~\ref{ax:lem:A4} with
\[
 m\left(\frac{n-2}{n}+2t\right)=\frac{n-2}{n},\qquad \sqrt{1-v}\le1-\frac v2,
\]
yields
\[
\begin{aligned}
\frac{Q(\phi,\phi)}
{\frac{4\alpha}{n(n-2)}
 \int_{-1}^1(1-x^2)(u')^2\,d\nu_n}
<{}&
-(1-v)\left(
f_n-\frac{2(1-\alpha)(n-2)}{n}
\right)-
v\left(
f_n-g_n-\frac{(1-\alpha)(n-2)}{n}
\right).
\end{aligned}
\]

By \eqref{ax:eq:fngn}, we have $Q(\phi,\phi)<0$ for all $n\geq 3$, a contradiction to the stability of $u$.
\end{proof}

Next, the first representation \eqref{ax:eq:25} of $Q(\phi,\phi)$ will be used to derive the following estimates.
\begin{lemma}\label{ax:lem:5.4}
Suppose $u$ is a positive non-constant stable solution to \eqref{ax:eq:10}. Under the same notations as in Lemma \ref{ax:lem:5.2}, we have
\begin{align*}
 -Q(\phi,\phi)
 &\ge
 a\left[(1-m)A_n(t)-\frac{D_n(t)}{\gamma_n}\right]+
 \frac{\alpha D_n(t)}{2\gamma_n}
 \left[
 (1-\omega)\|\sqrt{1+x}\,u\|_p^2
 +(1+\omega)\|\sqrt{1-x}\,u\|_p^2
 \right],
\end{align*}
where
\begin{equation}\label{ax:eq:31}
 \omega
 =
 -\frac{\gamma_nK_n(t)\int_{-1}^1x^3\,d\mu}
 {D_n(t)L_n(m)}.
\end{equation}
with
\[
 \gamma_n:=1-\alpha-\frac{\alpha}{n(n-2)}\in (0,\frac{3}{8}),
 \qquad
 L_n(m):=m+\frac{k(n-2)+1}{2kn}(1-m)>m
\]
and $k=n(\alpha^{-1}-1)$.

\end{lemma}
\begin{proof}
Recall from \eqref{ax:eq:8} and the Kazdan-Warner condition \eqref{ax:eq:KW} that
\begin{equation}\label{ax:eq:S1-xu}
 \begin{aligned}
\alpha\mathcal S(\sqrt{1-x}\,u,\sqrt{1-x}\,u)=a-bL_n(m)
 -\gamma_n\int_{-1}^1u^2\,d\nu_n.
 \end{aligned}
\end{equation}

Similarly, we can compute
\begin{equation}\label{ax:eq:S1+xu}
 \begin{aligned}
\alpha\mathcal S(\sqrt{1+x}\,u,\sqrt{1+x}\,u)=a+bL_n(m)
 -\gamma_n\int_{-1}^1u^2\,d\nu_n.
 \end{aligned}
\end{equation}

By Corollary \ref{ax:cor:mlower}, $m\geq \frac{2}{n+2}$ and equivalently $0<t\leq \frac{n-2}{4}$. By \eqref{ax:eq:24} and \eqref{ax:eq:26}, we obtain
\[
 |\omega|
 \le
 \frac{\gamma_n|K_n(t)|}{D_n(t)}
 \frac{m(1-m)}{L_n(m)}
 <
 2\gamma_n(1-m)
 <
 \frac34(1-m)
 <1.
\]

Combining \eqref{ax:eq:S1-xu} and \eqref{ax:eq:S1+xu}, we obtain
\begin{equation}\label{ax:eq:32}
 \begin{aligned}
 &\frac{\alpha(1-\omega)}2
 \mathcal S(\sqrt{1+x}\,u,\sqrt{1+x}\,u)
 +\frac{\alpha(1+\omega)}2
 \mathcal S(\sqrt{1-x}\,u,\sqrt{1-x}\,u)\\
 \qquad=\ &
 a-\gamma_n\int_{-1}^1u^2\,d\nu_n
 +\frac{\gamma_nK_n(t)}{D_n(t)}
 b\int_{-1}^1x^3\,d\mu.
 \end{aligned}
\end{equation}
Multiplying this identity by $D_n(t)/\gamma_n$ and comparing
with \eqref{ax:eq:25} yields
\begin{align*}
 -Q(\phi,\phi)=&a\left[(1-m)A_n(t)-\frac{D_n(t)}{\gamma_n}\right]\\
 &+
 \frac{\alpha D_n(t)}{2\gamma_n}
 \left[
 (1-\omega)S(\sqrt{1+x}\,u,\sqrt{1+x}\,u)
 +(1+\omega)S(\sqrt{1-x}\,u,\sqrt{1-x}\,u)
 \right]\\
 \geq& 
 a\left[(1-m)A_n(t)-\frac{D_n(t)}{\gamma_n}\right]+
 \frac{\alpha D_n(t)}{2\gamma_n}
 \left[
 (1-\omega)\|\sqrt{1+x}\,u\|_p^2
 +(1+\omega)\|\sqrt{1-x}\,u\|_p^2
 \right],
\end{align*}   
as desired, where the last inequality follows from the sharp Sobolev inequality \eqref{ax:eq:4}.

\end{proof}

Thus it remains to bound $\|\sqrt{1\pm x}\,u\|_p^2$. Since $d\mu=u^p\,d\nu_n$, these norms satisfy
\[
 \|\sqrt{1\pm x}\,u\|_p^2
 =
 \left(\int_{-1}^1(1\pm x)^{p/2}\,d\mu\right)^{2/p}.
\]
We next estimate these quantities by comparison with centered
two-point probability measures.

For fixed $m\in(0,1)$ and $|\theta|\le1-m$, let $x_{\pm}(\theta)=\frac{1}{2}(\theta\pm\sqrt{\theta^2+4m})\in [-1,1]$ and
\begin{equation*}
\mu_{\theta}=\frac{x_{+}(\theta)}{x_{+}(\theta)-x_{-}(\theta)}\delta_{x_{-}(\theta)}-\frac{x_{-}(\theta)}{x_{+}(\theta)-x_{-}(\theta)}\delta_{x_{+}(\theta)}.
\end{equation*}
Then we have
\[
 \int_{-1}^1x\,d\mu_\theta=0,\qquad
 \int_{-1}^1x^2\,d\mu_\theta=m,\qquad
 \int_{-1}^1x^3\,d\mu_\theta=m\theta.
\]

\begin{lemma}\label{ax:lem:5.1}
Let $n\ge4$, and let $\mu$ be a probability measure on $[-1,1]$
with zero first moment and second moment $m\in(0,1)$. Let $\mu_{\theta}$ be defined above. Define
\begin{equation}\label{ax:eq:Rmdef}
 \mathcal R_m(\theta)
 =
 \alpha
 \left(
 \int_{-1}^1(1+x)^{n/(n-2)}\,d\mu_\theta(x)
 \right)^{(n-2)/n}.    
\end{equation}

Then $\mathcal R_m$ is non-increasing and convex on
$[-(1-m),1-m]$, and is constant when $n=4$. Moreover,
\begin{equation}\label{ax:eq:33}
 \alpha
 \left(
 \int_{-1}^1(1\pm x)^{n/(n-2)}\,d\mu
 \right)^{(n-2)/n}
 \ge
 \mathcal R_m\left(
 \pm\frac1m\int_{-1}^1x^3\,d\mu
 \right).
\end{equation}
When $\theta=0$, we have
\begin{equation}\label{ax:eq:34}
 \mathcal R_m(0)
 \ge
 \alpha\sqrt{1+\frac{2m}{n-2}},
 \qquad n\ge6,
\end{equation}
\begin{equation}\label{ax:eq:35}
 \mathcal R_m(0)=\sqrt{\frac{1+m}{2}},
 \qquad n=4,
\end{equation}
and
\begin{equation}\label{ax:eq:36}
 \mathcal R_m(0)\ge\frac{3+m}{4},
 \qquad
 0\le-\frac{\mathcal R_m'(0)}m\le\frac3{40},
 \qquad n=5.
\end{equation}
\end{lemma}

We postpone the proof of this lemma to Appendix \ref{app:comparison-square-root}.

\subsection{The case \texorpdfstring{$n\geq6$}{n >= 6}}\label{sec:axial-6}

In this subsection, we always assume that $n\ge6$. We first derive an upper bound on the second
moment $m$ using only the Euler--Lagrange equation and
$a\le1$. This upper bound will then be combined with \eqref{ax:eq:27} to prove
$Q(\phi,\phi)<0$. 

We begin with the estimate \eqref{ax:eq:37} for probability measures. The proof of Lemma~\ref{ax:lem:6.1} is given in Appendix~\ref{app:comparison-affine}.

\begin{lemma}\label{ax:lem:6.1}
Let $n\ge6$, and let $\eta$ be a probability measure on $[-1,1]$
satisfying
\[
 \int_{-1}^1x\,d\eta=0,
 \qquad
 \int_{-1}^1x^2\,d\eta=m\in(0,1).
\]
Then
\begin{equation}\label{ax:eq:37}
 \frac{\alpha}{2}
 \left[
 \left(\int_{-1}^1(1+x)^p\,d\eta\right)^{2/p}
 +
 \left(\int_{-1}^1(1-x)^p\,d\eta\right)^{2/p}
 \right]
 >
 \alpha+(2-\alpha)m+\kappa_nm(1-m),
\end{equation}
where
 \begin{align}\label{estimate-on-kappa}
 \kappa_n=
 2-\alpha
 -\frac{2(n-2)}n
 \left(1-2^{\,1-2n/(n-2)}\right)>
 \frac n2
 \left(
 \frac{4\alpha}{n(n-2)}
 -\frac{2(1-\alpha)}n
 \right)
 >0.
 \end{align}
\end{lemma}

In the following lemma, we establish a upper bound of $m$.
\begin{lemma}\label{ax:lem:5.7}
For any $n\geq 6$, we have $m<\frac{n+3}{3(n+1)}$.
If we assume further that $m\geq \frac{2}{n}$, then
\begin{equation}
b\int_{-1}^{1}x^3\,d\mu>0,\qquad w>(1-\alpha)v.
\end{equation}
\end{lemma}
\begin{proof}
For $F=1\pm x$, the product computation based on \eqref{ax:eq:10} gives
\begin{equation}
 \begin{aligned}
 &\frac{\alpha}{2}
 \Bigl[
 \mathcal S((1+x)u,(1+x)u)
 +
 \mathcal S((1-x)u,(1-x)u)
 \Bigr]\\
 &\quad=
 a(1+m)+b\int_{-1}^1x^3\,d\mu-\left(1-\alpha-\frac{4\alpha}{n(n-2)}\right)\int_{-1}^1u^2\,d\nu_n\\
 &\qquad
 -\left(1-\alpha+\frac{4\alpha}{n(n-2)}\right)
 \int_{-1}^1x^2u^2\,d\nu_n.
 \end{aligned}
\end{equation}

Substituting \eqref{ax:eq:x2u2} and \eqref{ax:eq:22} into the identity above, we obtain
\begin{equation}\label{ax:eq:38}
 \begin{aligned}
 &\frac{\alpha}{2}
 \Bigl(
 \mathcal S((1+x)u,(1+x)u)
 +
 \mathcal S((1-x)u,(1-x)u)
 \Bigr)\\
 &\quad=
 a\left[
 \alpha+(2-\alpha)m
 +
 \left(
 \frac{4\alpha}{n(n-2)}
 -\frac{2(1-\alpha)}n
 \right)(1-m)
 \right]\\
 &\qquad
 -P_n\frac{4\alpha}{n(n-2)}
 \int_{-1}^1(1-x^2)^2(u')^2\,d\nu_n
 +R_nb\int_{-1}^1x^3\,d\mu.
 \end{aligned}
\end{equation}
The coefficients $P_n$ and $R_n$ are given explicitly by
\begin{align*}
 P_n&=\frac{k(n-2)^3+(3k^2+4)(n-2)^2+4k(n-2)-16}
 {(n+k)(n-2)^2(n+2-k)},\\
 R_n&=\frac{n(n^3-8n-16)+2k\{(n-2)(n^2+3n+4)+2n\}
 -k^2(n-2)(n+4)}{n(n+k)(n-2)(n+2-k)},
\end{align*}
and they satisfy
\begin{equation}\label{ax:eq:39}
 0<R_n<\frac75,
 \qquad
 P_n>(1-\alpha)R_n>0.
\end{equation}

Furthermore, we apply the sharp Sobolev inequality \eqref{ax:eq:4} to $(1\pm x)u$ in \eqref{ax:eq:38}, and then use \eqref{ax:eq:37} in Lemma~\ref{ax:lem:6.1} together with $a\le1$ to get
\begin{equation}\label{ax:eq:40}
 \begin{aligned}
 &P_n\frac{4\alpha}{n(n-2)}
 \int_{-1}^1(1-x^2)^2(u')^2\,d\nu_n
 -R_nb\int_{-1}^1x^3\,d\mu<
 a(1-m)
 \left[
 \frac{4\alpha}{n(n-2)}
 -\frac{2(1-\alpha)}n
 -\kappa_nm
 \right].
 \end{aligned}
\end{equation}

If $m\ge \frac{2}{n}$, the right-hand side is negative by \eqref{estimate-on-kappa}. Using the definitions of $v$ and $w$ in \eqref{ax:eq:23}, 
\eqref{ax:eq:40} implies $P_nv-R_nw<0$.
Together with \eqref{ax:eq:39}, it gives $R_nw>P_nv>(1-\alpha)R_nv$.
Hence
\begin{equation}\label{ax:eq:41}
 w>(1-\alpha)v>0,
 \qquad
 b\int_{-1}^1x^3\,d\mu>0.
\end{equation}

On the other hand, since $0<R_n<\frac{7}{5}$ by \eqref{ax:eq:39}, \eqref{ax:eq:40} yields
\begin{equation*}
b\int_{-1}^1x^3\,d\mu>\frac{5}{7}
 a(1-m)
 \left[
 \frac{4\alpha}{n(n-2)}
 -\frac{2(1-\alpha)}n
 \right]\left(\frac{nm}{2}-1\right)\geq 0.
\end{equation*}

By \eqref{ax:eq:41}, $b$ and $\int_{-1}^1x^3\,d\mu$ have the same
sign. Applying the sharp Sobolev inequality \eqref{ax:eq:4}, \eqref{ax:eq:S1-xu}, \eqref{ax:eq:S1+xu}, \eqref{ax:eq:33} and the monotonicity of $\mathcal{R}_m$ from Lemma~\ref{ax:lem:5.1} for $b\geq 0$ and $b<0$ respectively, we obtain
\begin{equation}\label{ax:eq:42}
 a-\gamma_n\int_{-1}^1u^2\,d\nu_n
 -|b|L_n(m)
 \ge \mathcal R_m(0).
\end{equation}

Note that \eqref{ax:eq:22}, together with \eqref{ax:eq:2}, can be rewritten as
\begin{equation}\label{ax:eq:43}
 \begin{aligned}
 \int_{-1}^1u^2\,d\nu_n
 =&
 a\frac{n+1}{n}(1-m)+
 \frac{2\bigl(k(n-2)+3\bigr)}
 {(n-2)(n+2-k)}
 \frac{4\alpha}{n(n-2)}
 \int_{-1}^1(1-x^2)^2(u')^2\,d\nu_n\\
 &-
 \frac{b\int_{-1}^1x^3\,d\mu}{n\lambda_n}.
 \end{aligned}
\end{equation}
Since \eqref{ax:eq:24} and \eqref{ax:eq:41} imply
\[
 |b|
 \ge
 \frac{b\int_{-1}^1x^3\,d\mu}{m(1-m)},
\]
we may insert \eqref{ax:eq:43} into \eqref{ax:eq:42} and obtain that
\[
 \mathcal R_m(0)
 \le
 1-\frac{\gamma_n(n+1)}n(1-m)
 -
 \left(
 \frac{L_n(m)}{m(1-m)}
 -\frac{\gamma_n}{n\lambda_n}
 \right)
 \frac{b\int_{-1}^1x^3\,d\mu}{a}.
\]

The coefficient estimate \eqref{ax:eq:C1}, proved in
Lemma~\ref{ax:lem:D1}, states that
\[
 (1-m)
 \left(
 \frac{L_n(m)}{m(1-m)}
 -\frac{\gamma_n}{n\lambda_n}
 \right)>1.
\]
Combining this with \eqref{estimate-on-kappa}, \eqref{ax:eq:40} and \eqref{ax:eq:39}, we arrive at the necessary inequality when $m\ge\frac2n$
\begin{equation}\label{ax:eq:44}
 \begin{aligned}
 \mathcal R_m(0)
 <
 1-\frac{\gamma_n(n+1)}n(1-m)-
 \frac57
 \left(
 \frac{4\alpha}{n(n-2)}
 -\frac{2(1-\alpha)}n
 \right)
 \left(\frac{nm}{2}-1\right).
 \end{aligned}
\end{equation}

We now combine the upper bound \eqref{ax:eq:44} with
the lower bound \eqref{ax:eq:34} on $\mathcal R_m(0)$. Define
\[
 \Psi_n(m)
 :=
 \alpha\sqrt{1+\frac{2m}{n-2}}
 -
 \left[
 1-\frac{\gamma_n(n+1)}n(1-m)
 -\frac57
 \left(
 \frac{4\alpha}{n(n-2)}-\frac{2(1-\alpha)}n
 \right)
 \left(\frac{nm}{2}-1\right)
 \right].
\]
It's not hard to verify that the function $\Psi_n$ is concave in $m$. The proof of Lemma~\ref{ax:lem:D1}
verifies the endpoint inequalities
\[
 \Psi_n\left(\frac{n+3}{3(n+1)}\right)>0,
 \qquad
 \Psi_n(1)>0.
\]
By concavity, $\Psi_n(m)>0$ for every $ m\in\left[\frac{n+3}{3(n+1)},1\right]$.
Hence, throughout this interval, the lower bound \eqref{ax:eq:34} is
strictly bigger than the upper bound \eqref{ax:eq:44},
a contradiction. Therefore, when $n\ge6$
\begin{equation}\label{ax:eq:45}
 m<\frac{n+3}{3(n+1)}.
\end{equation}
\end{proof}

We now prove instability of $u$ when $n\geq 6$. 
\begin{proof}[Proof of Theorem \ref{ax:thm:4.1} when $n\geq 6$]

We divide the argument according to whether
$m\le2/n$ or $m>2/n$.

Suppose first that $m\le2/n$. By the definition of $t$ in \eqref{ax:eq:12}, we have
\[
 t=\frac{(n-2)(1-m)}{2nm}
 \ge\frac{(n-2)^2}{4n}.
\]
Combining $|w|<2(1-\alpha)m\sqrt{1-v}$ from Lemma~\ref{ax:lem:4.2} with \eqref{ax:eq:h-bound} in Lemma~\ref{ax:lem:A4}, namely,
\[
 0<-h_n(t)<\frac{n-2}{n}+2t,
 \qquad
 m\left(\frac{n-2}{n}+2t\right)=\frac{n-2}{n},
\]
we obtain
\[
 h_n(t)w
 <
 \frac{2(1-\alpha)(n-2)}{n}\sqrt{1-v}.
\]
Since $\sqrt{1-v}\le5/4-v$,
\eqref{ax:eq:27} yields
\[
\begin{aligned}
\frac{Q(\phi,\phi)}
{\frac{4\alpha}{n(n-2)}
 \int_{-1}^1(1-x^2)(u')^2\,d\nu_n}
<{}&
-(1-v)\left(
f_n-\frac{5(1-\alpha)(n-2)}{2n}
\right)\\
&-
v\left(
f_n-g_n-\frac{(1-\alpha)(n-2)}{2n}
\right).
\end{aligned}
\]

The two inequalities in \eqref{ax:eq:fngnn6} show that both quantities in brackets
are positive whenever
\[
 t\ge\frac{(n-2)^2}{4n},
 \qquad n\ge6.
\]
Hence $Q(\phi,\phi)<0$ in this range.

We now consider the remaining range
\[
 \frac2n<m<\frac{n+3}{3(n+1)}.
\]
By \eqref{ax:eq:41}, we obtain that $w>(1-\alpha)v$.
Since $h_n(t)<0$ by \eqref{ax:eq:h-bound}, \eqref{ax:eq:27} therefore gives
\[
\frac{Q(\phi,\phi)}
{\frac{4\alpha}{n(n-2)}
 \int_{-1}^1(1-x^2)(u')^2\,d\nu_n}
<
-f_n(t)+\bigl(g_n(t)+(1-\alpha)h_n(t)\bigr)v.
\]
The bound \eqref{ax:eq:45} implies $t>\frac{n-2}{n+3}$.
On this range, the estimate \eqref{ax:eq:fngnhn} gives
\[
 f_n(t)-g_n(t)-(1-\alpha)h_n(t)>0.
\]
Since $f_n(t)>0$ and $0<v<1$, we may write
\[
 -f_n+\bigl(g_n+(1-\alpha)h_n\bigr)v
 =
 (1-v)(-f_n)
 +v\bigl[-f_n+g_n+(1-\alpha)h_n\bigr],
\]
which is a convex combination of two strictly negative quantities.
Thus $Q(\phi,\phi)<0$ also in this range.
Together with Corollary \ref{ax:cor:mlower}, this
proves instability of $u$ for every $n\ge6$.
    
\end{proof}

\subsection{The case $n=4,5$ and $n=3$ with $m\geq\frac{4}{7}$}

In this subsection, We prove the instability of $u$ in dimensions four and five, together with the large-moment $m\geq\frac{4}{7}$ in three dimension. Recall the weight $\omega$ defined in \eqref{ax:eq:31} and set
\begin{equation}\label{ax:eq:weighted-root-norm}
 \mathcal N_\omega
 :=
 \frac{\alpha}{2}\left[
 (1-\omega)\|\sqrt{1+x}\,u\|_p^2
 +(1+\omega)\|\sqrt{1-x}\,u\|_p^2
 \right].
\end{equation}

We first estimate $\mathcal N_\omega$ in each dimension.
\begin{lemma}\label{ax:lem:5.8}
Let $\mathcal N_\omega$ be given in \eqref{ax:eq:weighted-root-norm}. Then we have
\begin{equation*}
\mathcal N_\omega\geq 1-c_n(1-m)>0
\end{equation*}
with $c_3=\frac{2}{7}$ when  $m\geq\frac{4}{7}$,  $c_4=\frac{5}{18}$ and $c_5=\frac{33}{128}$.
\end{lemma}
\begin{proof}
We first consider $n=4,5$. By Corollary \ref{ax:cor:mlower}, it remains to treat
\[
 m>\frac{2}{n+2},
 \qquad\text{equivalently}\qquad
 0<t<\frac{n-2}{4}.
\]

Recall the definition of $\mathcal{R}_m$ from \eqref{ax:eq:Rmdef}. Applying \eqref{ax:eq:33} to the two square-root norms and using the convexity of $\mathcal R_m$ from Lemma~\ref{ax:lem:5.1}, we obtain
\begin{align*}
 \mathcal N_\omega
 &\ge
 \frac{1-\omega}{2}\mathcal R_m\left(\frac1m\int_{-1}^1x^3\,d\mu\right)
 +\frac{1+\omega}{2}\mathcal R_m\left(-\frac1m\int_{-1}^1x^3\,d\mu\right)\\
 &\ge
 \mathcal R_m\left(-\omega \frac1m\int_{-1}^1x^3\,d\mu\right)=
 \mathcal R_m\left(
 \frac{\gamma_nK_n(t)
 \bigl(\int_{-1}^1x^3\,d\mu\bigr)^2}
 {D_n(t)mL_n(m)}
 \right).
\end{align*}

When $n=4$, the exponent $n/(n-2)=2$. Since every
$\mu_\theta$ has zero first moment and second moment $m$,
\[
 \int_{-1}^1(1+x)^2\,d\mu_\theta=1+m.
\]
Hence $\mathcal R_m$ is independent of $\theta$, and
\[
 \mathcal N_\omega
 \ge
 \mathcal R_m(0)
 =
 \alpha(1+m)^{1/2}
 =
 \sqrt{\frac{1+m}{2}},
\]
because $\alpha=2^{-1/2}$. Moreover, since $m\ge1/3$, a direct calculation gives
\[
 1-\frac{1-m}{2}
 -
 \left(1-\frac{5(1-m)}{18}\right)^2
 =
 \frac{(1-m)(18-25(1-m))}{324}\ge0,
\]
Therefore
\begin{equation*}
 \mathcal N_\omega
 \ge
 1-\frac5{18}(1-m),
 \qquad n=4.
\end{equation*}

For $n=5$, convexity of $ \mathcal R_m(\theta)$ gives
\[
 \mathcal R_m(\theta)
 \ge
 \mathcal R_m(0)+\mathcal R_m'(0)\theta.
\]
Using \eqref{ax:eq:24}, $L_n(m)>m$ from Lemma~\ref{ax:lem:5.4}, and
$|K_n(t)|<2D_n(t)$ from \eqref{ax:eq:26}, we have
\[
 \left|
 \frac{\gamma_nK_n(t)
 \bigl(\int_{-1}^1x^3\,d\mu\bigr)^2}
 {D_n(t)mL_n(m)}
 \right|
 <
 2\gamma_n(1-m)^2.
\]
Together with \eqref{ax:eq:36}, this yields
\[
 \mathcal N_\omega
 \ge
 \frac{3+m}{4}
 -\frac{3\gamma_n}{20}m(1-m)^2.
\]
Since $\gamma_n<1/5$ and $m(1-m)\le1/4$,
\[
 \frac{3\gamma_n}{20}m(1-m)^2
 <
 \frac3{400}(1-m).
\]
Hence
\[
 \mathcal N_\omega
 >
 1-\frac{103}{400}(1-m)
 \ge
 1-\frac{33}{128}(1-m),
\]
and therefore
\begin{equation}
 \mathcal N_\omega
 \ge
 1-\frac{33}{128}(1-m),
 \qquad n=5.
\end{equation}

We now turn to $n=3$. In this subsection
we treat the large-moment range
\[
 \frac47\le m<1,
 \qquad\text{equivalently}\qquad
 0<t\le\frac18.
\]
The intermediate interval $2/5<m<4/7$ will be handled separately
in the next subsection.

In the range $\frac47\le m<1$, the last assertion of Lemma~\ref{ax:lem:A1} gives
\[
 |K_3(t)|<D_3(t),
\]
so the weights \eqref{ax:eq:31} are admissible. Since $p=6$, the norms of the square-root tests are determined exactly by
\[
 \int_{-1}^1(1\pm x)^3\,d\mu
 =
 1+3m\pm\int_{-1}^1x^3\,d\mu.
\]
A direct comparison of  two cubic-root expressions gives
\begin{equation}\label{ax:eq:49}
 \begin{aligned}
 \mathcal N_\omega
 =&\frac{\alpha}{2}\biggl[
 (1-\omega)
 \left(1+3m+\int_{-1}^1x^3\,d\mu\right)^{1/3}+
 (1+\omega)
 \left(1+3m-\int_{-1}^1x^3\,d\mu\right)^{1/3}
 \biggr]\\
 &\ge 1-\frac27(1-m).
 \end{aligned}
\end{equation}
The estimate \eqref{ax:eq:49} is proved in
Lemma~\ref{ax:lem:B1}.   
\end{proof}

We now insert the estimate in Lemma~\ref{ax:lem:5.8} into Lemma~\ref{ax:lem:5.4} to prove Theorem \ref{ax:thm:4.1} when $n=4,5$ and $n=3$ with $m\geq\frac{4}{7}$. 
\begin{proof}[Proof of Theorem \ref{ax:thm:4.1} when $n=4,5$ and $n=3$ with $m\geq\frac{4}{7}$]
 
By the sharp Sobolev inequality \eqref{ax:eq:4} and \eqref{ax:eq:32},
\[
 a-\gamma_n\int_{-1}^1u^2\,d\nu_n
 +\frac{\gamma_nK_n(t)}{D_n(t)}
 b\int_{-1}^1x^3\,d\mu
 \ge
 \mathcal N_\omega.
\]
Since $0<a\le1$ and $\mathcal N_\omega>0$, Lemma~\ref{ax:lem:5.8} and division by $a$ give
\[
 1-\frac{\gamma_n}{a}\int_{-1}^1u^2\,d\nu_n
 +\frac{\gamma_nK_n(t)}{aD_n(t)}
 b\int_{-1}^1x^3\,d\mu
 \ge
 1-c_n(1-m).
\]
Equivalently,
\[
 -\frac{D_n(t)}{a}\int_{-1}^1u^2\,d\nu_n
 +\frac{K_n(t)}{a}
 b\int_{-1}^1x^3\,d\mu
 \ge
 -\frac{c_nD_n(t)}{\gamma_n}(1-m).
\]
Substituting this estimate into \eqref{ax:eq:25} yields
\begin{equation}\label{ax:eq:low-dimensional-second-variation}
 -\frac{Q(\phi,\phi)}{a}
 \ge
 \frac{1-m}{\gamma_n}
 \bigl[\gamma_nA_n(t)-c_nD_n(t)\bigr].
\end{equation}

For $n=4,5$, it becomes
\begin{equation}
 -\frac{Q(\phi,\phi)}a
 \ge
 \frac{1-m}{\gamma_n}
 \begin{cases}
 \displaystyle
 \gamma_nA_4(t)-\frac5{18}D_4(t),
 & n=4,\\[6pt]
 \displaystyle
 \gamma_nA_5(t)-\frac{33}{128}D_5(t),
 & n=5.
 \end{cases}
\end{equation}
The estimates \eqref{ax:eq:A2} in Lemma~\ref{ax:lem:A2} prove
$\gamma_4A_4(t)-\frac5{18}D_4(t)>0$ for $0\le t\le\frac12$ and
$\gamma_5A_5(t)-\frac{33}{128}D_5(t)>0$ for
$0\le t\le\frac34$.
Hence $Q(\phi,\phi)<0$ in dimensions four and five.

For $n=3$ and $0<t\le1/8$, \eqref{ax:eq:low-dimensional-second-variation}
gives
\[
 -\frac{Q(\phi,\phi)}a
 \ge
 \frac{1-m}{\gamma_3}
 \left[
 \gamma_3A_3(t)-\frac27D_3(t)
 \right].
\]
The estimate \eqref{ax:eq:A2} in Lemma~\ref{ax:lem:A2} proves that the quadratic
$\gamma_3A_3(t)-\frac27D_3(t)$ is positive on $0\le t\le\frac18$:
\begin{equation}
 \gamma_3A_3(t)-\frac27D_3(t)>0,
 \qquad
 0\le t\le\frac18.
\end{equation}
Therefore $Q(\phi,\phi)<0$ throughout the three-dimensional case with $m\geq \frac{4}{7}$.   
\end{proof}

\subsection{The case when  $n=3$ and $\frac25<m<\frac47$}

It remains only to consider $n=3$ and $\frac25<m<\frac47$. For convenience, the estimates below are established on the closed interval $2/5\le m\le4/7$. We first record the comparison for the
translated functions $\tau\pm x$ that will be used in this range.

\begin{lemma}\label{ax:lem:8.1}
Let $\tau\ge1$, and define
\[
 \mathcal A_{m,\tau}(\theta)
 =
 \alpha
 \left(
 \int_{-1}^1(\tau+x)^6\,d\mu_\theta(x)
 \right)^{1/3},
 \qquad |\theta|\le1-m.
\]
Then $\mathcal A_{m,\tau}$ is increasing on $[-(1-m),1-m]$
and strictly convex in its interior.

For every $\lambda\in[0,1]$,
\begin{align}
 \alpha\biggl[\lambda
 \left(\int_{-1}^1(\tau+x)^6\,d\mu\right)^{1/3}
 +(1-\lambda)
 \left(\int_{-1}^1(\tau-x)^6\,d\mu\right)^{1/3}
 \biggr]\ge
 \mathcal A_{m,\tau}\left((2\lambda-1)\frac1m\int_{-1}^1x^3\,d\mu\right).
 \label{ax:eq:translated-comparison}
\end{align}
Moreover,
\[
 \mathcal A_{m,\tau}(0)
 =
 \left(
 \frac{\tau^6+15\tau^4m+15\tau^2m^2+m^3}{4}
 \right)^{1/3}.
\]
\end{lemma}
The proof of Lemma~\ref{ax:lem:8.1} is postponed to Appendix~\ref{app:comparison-translated}.
We now give the proof of Theorem \ref{ax:thm:4.1} when $n=3$ and $\frac25\leq m\leq\frac47$.
\begin{proof}[Proof of Theorem \ref{ax:thm:4.1} when $n=3$ and $\frac25\leq m\leq\frac47$]
Choose $\tau^2=\frac92$. Applying the product computation based on \eqref{ax:eq:10} with
$F(x)=\tau\pm x$, and using \eqref{ax:eq:KW} and \eqref{ax:eq:x2u2}, gives
\begin{equation}\label{ax:eq:47}
 \begin{aligned}
 \alpha\mathcal S((\tau\pm x)u,(\tau\pm x)u)
 ={}&
 aT(m)
 -\left(5-\frac{17\alpha}{3}\right)
 \int_{-1}^1u^2\,d\nu_n+
 \frac{21-5\alpha}{36}
 b\int_{-1}^1x^3\,d\mu
 \pm
 2\tau b\left(m+\frac{1-m}{6}\right),
 \end{aligned}
\end{equation}
where
\[
 T(m)=\frac92+m+\frac{3+\alpha}{9}(1-m).
\]

On $1/8\le t\le1/4$, Lemma~\ref{ax:lem:A1} and \eqref{ax:eq:E1} in Lemma~\ref{ax:lem:A3} give the two
coefficient inequalities
\[
 0<K_3(t)<2D_3(t),
\]
and
\[
 \left(5-\frac{17\alpha}{3}\right)K_3(t)
 -
 \frac{21-5\alpha}{36}D_3(t)>0,
\]
the second being \eqref{ax:eq:E1}.
We combine two identities in \eqref{ax:eq:47}
with coefficients $(1-\omega)/2$ and $(1+\omega)/2$, where
\[
 \omega
 =
 \frac{
 \left[
 \frac{21-5\alpha}{36}
 -
 \left(5-\frac{17\alpha}{3}\right)
 \frac{K_3(t)}{D_3(t)}
 \right]
 \int_{-1}^1x^3\,d\mu
 }
 {2\tau\left(m+\frac{1-m}{6}\right)}.
\]
These coefficients are positive. Indeed,
\[
 |\omega|
 <
 \frac34(1-m)
 \le\frac9{20}<1,
\]
using $K_3<2D_3$ from Lemma~\ref{ax:lem:A1},
$5-17\alpha/3<3/2$, and $\tau>2$.

The choice of $\omega$ cancels the term
$b\int_{-1}^1x^3\,d\mu$ in the second-variation formula.
Set
\[
 \theta_0=\frac1m\int_{-1}^1x^3\,d\mu.
\]
The estimate \eqref{ax:eq:translated-comparison} in Lemma~\ref{ax:lem:8.1}, with weight $(1-\omega)/2$, gives the lower
bound $\mathcal A_{m,\tau}(-\omega\theta_0)$. Since $\omega$ has the
opposite sign to $\theta_0$ and $\mathcal A_{m,\tau}$ is increasing,
\[
 \mathcal A_{m,\tau}(-\omega\theta_0)
 \ge \mathcal A_{m,\tau}(0)
 =
 \left(
 \frac{\tau^6+15\tau^4m+15\tau^2m^2+m^3}{4}
 \right)^{1/3}.
\]
We denote the last quantity by $F(m)$.

Using the sharp Sobolev inequality \eqref{ax:eq:4} in
\eqref{ax:eq:47}, and then substituting 
into \eqref{ax:eq:25}, we obtain
\[
 -\frac{Q(\phi,\phi)}a
 \ge
 (1-m)A_3(t)
 -
 \frac{D_3(t)}{5-17\alpha/3}
 \left(T(m)-\frac{F(m)}a\right).
\]
Since $F(m)>0$ and $a\le1$,
\[
 \frac{F(m)}a\ge F(m),
\]
and therefore
\[
 -\frac{Q(\phi,\phi)}a
 \ge
 (1-m)A_3(t)
 -
 \frac{D_3(t)}{5-17\alpha/3}
 \bigl(T(m)-F(m)\bigr).
\]

The function $F$ is strictly concave on $[2/5,4/7]$.
The proof of \eqref{ax:eq:Flower} in Lemma~\ref{ax:lem:A3} gives the endpoint bounds
\[
 F\left(\frac25\right)>\frac{191}{50},
 \qquad
 F\left(\frac47\right)>\frac{83}{20}.
\]
Hence the chord joining these two lower bounds gives
\[
 F(m)>
 \frac{61}{20}+\frac{77}{40}m,
 \qquad
 \frac25\le m\le\frac47.
\]
Substituting $ m=\frac1{1+6t}$
and multiplying by $1+6t$, it suffices to prove
\begin{equation}\label{ax:eq:48}
 (1+6t)
 \left[
 \left(5-\frac{17\alpha}{3}\right)
 (1-m)A_3(t)
 -
 D_3(t)
 \left(
 T(m)-\frac{61}{20}-\frac{77m}{40}
 \right)
 \right]>0,
 \qquad
 m=\frac1{1+6t}.
\end{equation}
After this substitution, the left-hand side of \eqref{ax:eq:48} is a cubic polynomial
in $t$. The proof of \eqref{ax:eq:A3D3} in Lemma~\ref{ax:lem:A3} shows that its second derivative
is negative on $[1/8,1/4]$ and that its values at $t=1/8$ and
$t=1/4$ are positive. Thus
$Q(\phi,\phi)<0$ throughout the remaining three-dimensional
interval.

Combining this with the preceding cases proves
$Q(\phi,\phi)<0$
for every non-constant solution of \eqref{ax:eq:10}. Hence a
constrained stable solution must be constant. The normalization
then gives $u\equiv1$, and the Euler--Lagrange equation yields
$a=1$ and $b=0$. This completes the proof of
Theorem~\ref{ax:thm:4.1}.    
\end{proof}

\section{The axially-symmetric inequality, equality cases, and sharpness}
\label{sec:axial-9}
\begin{proof}[Proof of Theorem~\ref{thm:axial}]
Let $(u,r_+,r_-)\in \mathcal{L}_{\mathrm{rel}}$ be a minimizer of $\mathcal{E}_{\mathrm{rel}}$, whose existence follows
from Proposition~\ref{ax:prop:2.2}. Then
\begin{equation*}
\mathcal{E}_{\mathrm{rel}}(u,r_+,r_-)\leq \mathcal{E}_{\mathrm{rel}}(1,0,0)=1.
\end{equation*}

Propositions~\ref{ax:prop:2.2} and \ref{ax:prop:3.1} imply that either $(u,r_+,r_-)=(0,\frac{1}{2},\frac{1}{2})$ or $u\in \mathcal{L}_r$ is smooth, strictly positive with $r_+=r_-=0$. As a constrained minimizer, $u$ is stable, so
Theorem~\ref{ax:thm:4.1} gives $u\equiv1$.
Thus the relaxed infimum is one. 

For any nonzero real-valued axially-symmetric function $f\in \mathcal{L}_r$, $|f|/\|f\|_p$ is admissible for the
normalized problem. Homogeneity and the Sobolev chain rule give
\[
 1\le E\left(\frac{|f|}{\|f\|_p}\right)
 =\frac{E(f)}{\|f\|_p^2},
\]
which proves \eqref{ax:eq:1}. If such an $f$ attains equality,
then $(|f|/\|f\|_p,0,0)\in \mathcal{L}_{\mathrm{rel}}$ is a minimizer of $\mathcal{E}_{\mathrm{rel}}$, and the
preceding argument shows that $|f|$ is constant. The Sobolev
chain rule then gives
\[
 |\nabla f|=|\nabla|f||=0
 \qquad\text{almost everywhere}.
\]
Since $\mathbb S^n$ is connected, $f$ is constant. Conversely,
every constant function, including the zero function, attains
equality.

The optimality of $\alpha$ follows from Remark \ref{rem:sharpness}. Hence, we complete the proof of Theorem \ref{thm:axial}.
\end{proof}

\appendix

\section{Coefficient estimates}
We collect the algebraic estimates used in the proof of instability of Theorem \ref{thm:axial}. Recall the coefficients $A_n$, $D_n$, $K_n$ in \eqref{ax:eq:canonical-coefficients}, $f_n$, $g_n$, $h_n$ in \eqref{ax:eq:gradient-coefficients} and
\[
 k=n(\alpha^{-1}-1),\qquad
 \alpha=2^{-\frac{2}{n}}=\frac{n}{n+k}.
\]
Then we have
\begin{equation}\label{ax:eq:52}
 \frac43<k<
 \begin{cases}
    2,&\text{ if }n\geq 3,\\
    \frac{5}{3},&\text{ if }n\geq 4.
 \end{cases}
\end{equation}

\subsection{Estimates of coefficients in two second-variation representations}

\begin{lemma}\label{ax:lem:A1}
For $n\geq 3$, the coefficients $A_n$, $D_n$, $K_n$ in \eqref{ax:eq:canonical-coefficients} satisfy $A_n(t), D_n(t)>0$ for any $0\leq t\leq \frac{n-2}{2}$ and
\begin{equation*}
-D_n(t)<K_n(t)<2D_n(t), \text{ for any }0<t\leq \frac{n-2}{4}.
\end{equation*}

If we assume further that $n=3$ and $0\leq t\leq \frac{1}{8}$, then
\begin{equation*}
|K_3(t)|<D_3(t).
\end{equation*}
\end{lemma}
\begin{proof}
Since the function
\begin{equation*}
(1-\alpha)(3+4t)-\frac{4\alpha t^2}{n(n-2)}
\end{equation*}
is concave and positive at $t=0$ and $t=\frac{n-2}{2}$, we have $A_n>0$. 

The expression of $D_n$ also gives
\begin{equation*}
D_n(t)\geq (1-\alpha)\left(\frac{1}{2}+2t\right)+\frac{2\alpha t^2}{n(n-2)}>0.
\end{equation*}

Meanwhile, since $\frac{3}{4}<k<2$, $D_n+K_n$ is a quadratic polynomial in $t$ with positive coefficients and $2D_n-K_n$ is concave and positive at $t=0$ and $t=\frac{n-2}{4}$. It follows that
\begin{equation*}
-D_n(t)<K_n(t)<2D_n(t)
\end{equation*}
for any $0\leq t \leq \frac{n-2}{4}$.

If $n=3$ and $0\leq t\leq \frac{1}{8}$, then $D_3-K_3$ is concave and positive at $t=0$ and $t=\frac{1}{8}$. Then we have $D_3(t)>K_3(t)$ as desired.
\end{proof}

\begin{lemma}\label{ax:lem:A4}
For $n\geq 3$, we have
\begin{equation}\label{ax:eq:fngnn4}
f_n(t)>\frac{2(1-\alpha)(n-2)}{n},\qquad f_n(t)-g_n(t)>\frac{(1-\alpha)(n-2)}{n}\qquad \text{ if }t\geq \frac{n-2}{4}
\end{equation}
and
\begin{equation}\label{ax:eq:h-bound}
0<-h_n(t)<\frac{n-2}{n}+2t\qquad \text{ if }0\leq t\leq \frac{n-2}{2}.
\end{equation}

For $n\geq 6$, we have
\begin{equation}\label{ax:eq:fngnn6}
f_n(t)>\frac{5(1-\alpha)(n-2)}{2n},\qquad f_n(t)-g_n(t)>\frac{(1-\alpha)(n-2)}{2n}\qquad \text{ if }t\geq \frac{(n-2)^2}{4n}
\end{equation}
and
\begin{equation}\label{ax:eq:fngnhn}
f_n(t)-g_n(t)-(1-\alpha)h_n(t)>0\qquad \text{ if }t\geq \frac{n-2}{n+3}.
\end{equation}
\end{lemma}
\begin{proof}
For $n\geq 4$, since $f_n$ and $f_n-g_n$ are increasing for $t\geq \frac{n-2}{4}$ and $\frac{4}{3}<k<\frac{9}{5}$, plugging $t=\frac{n-2}{4}$ into $f_n$ and $f_n-g_n$ yields \eqref{ax:eq:fngnn4}.

When $n=3$, we have $\alpha=2^{-2/3}>5/8$ and
\[
 f_3(t)=4\alpha t^2+2(1-\alpha)t,
 \qquad
 g_3(t)=\frac{k-20t^2}{5-k},
 \qquad
 k=3(\alpha^{-1}-1)<2.
\]
Thus $f_3$ and $f_3-g_3$ are strictly increasing for $t\geq1/4$.
At the left endpoint,
\[
 f_3\left(\frac14\right)-\frac{2(1-\alpha)}3
 =
 \frac{5\alpha-2}{12}>0,
\]
and
\[
 f_3\left(\frac14\right)-g_3\left(\frac14\right)
 -\frac{1-\alpha}{3}
 =
 \frac{4\alpha^2+32\alpha-21}{6(8\alpha-3)}>0,
\]
where the last inequality follows from $\alpha>5/8$.
This proves \eqref{ax:eq:fngnn4} also for $n=3$.
The convexity and endpoint identities for $h_n$ used below
remain valid in this dimension.

Since $h_n$ is a convex quadratic formula in $t$ with
\begin{equation*}
-\frac{n-2}{n}<h_n(0)<0,\qquad h_n'(0)>-2,\qquad h_n\left(\frac{n-2}{2}\right)
 =
 -\frac{(n-2)(n+2)}{2n}<0,
\end{equation*}
it follows that
\begin{equation*}
 0<-h_n(t)<\frac{n-2}{n}+2t,
 \qquad
 0\le t\le\frac{n-2}{2}.
\end{equation*}

For $n\geq 6$, the functions $f_n$ and $f_n-g_n$ are increasing on the range in \eqref{ax:eq:fngnn6}. Their values at its left endpoint satisfy the two lower bounds in that formula, so the bounds hold throughout the range.

For \eqref{ax:eq:fngnhn}, bound the coefficients of $f_n-g_n-(1-\alpha)h_n$ using \eqref{ax:eq:52}. This gives
\[
 \frac{n+2-k}{1-\alpha}
 \bigl(f_n-g_n-(1-\alpha)h_n\bigr)
 >
 \frac{12(n-1)n(n+2)}{5(n-2)^2}t^2
 +7t-2(n-2)-\frac53.
\]
The right-hand side is increasing for $t\geq 0$ and is positive at $t=\frac{n-2}{n+3}$. Then \eqref{ax:eq:fngnhn} follows.
\end{proof}


\subsection{Estimates of \texorpdfstring{$L_n$ and $\Psi_n$}{L and Psi estimates}}

We verify the two estimates of $L_n$ and $\Psi_n$ used in the proof of Lemma \ref{ax:lem:5.7}.
\begin{lemma}\label{ax:lem:D1}
Under the notation of Lemma \ref{ax:lem:5.7}, for $n\geq 6$, we have
\begin{equation}\label{ax:eq:C1}
(1-m)\left(
\frac{L_n(m)}{m(1-m)}-\frac{\gamma_n}{n\lambda_n}
\right)>1,\qquad \text{ for any }0<m<1
\end{equation}
and
\[
\Psi_n(m)>0,
\qquad
\text{ for any }
\frac{n+3}{3(n+1)}\le m\le1.
\]
\end{lemma}
\begin{proof}
 First, for $n\ge6$, we have $n\lambda_n>1/2$. Indeed, using
$k<5/3$ from \eqref{ax:eq:52},
\[
4n(n+2-k)-\{3n^2+8n+8-k(n+4)\}
=n^2-k(3n-4)-8>n^2-5n-\frac43>0.
\]
Hence $\gamma_n/(n\lambda_n)<2\gamma_n$. Since
\[
L_n(m)=m+\frac{k(n-2)+1}{2kn}(1-m),
\]
the coefficient of $1-m$ satisfies
\[
\frac{k(n-2)+1}{2kn}>\frac{\gamma_n}{n\lambda_n}.
\]
After applying the bounds in \eqref{ax:eq:52}, this comparison reduces to
\[
15n^3-126n^2+217n+70>0.
\]
Writing $q=n-6\ge0$, the left-hand side becomes
$15q^3+144q^2+325q+76>0$. Therefore, since $0<m<1$, \eqref{ax:eq:C1} follows.

We next verify the two endpoint inequalities for the concave function
$\Psi_n$ introduced after \eqref{ax:eq:44}. At $m=1$, the inequality
$\Psi_n(1)>0$ is equivalent, after multiplication by the positive
factor $n+k$, to
\[
n\left(\sqrt{\frac{n}{n-2}}-1\right)
>
\frac{12k-10-10k/n}{7}.
\]
The left-hand side is greater than one, while \eqref{ax:eq:52} gives
\[
\frac{12k-10-10k/n}{7}
<
\frac{34/5+1/15}{7}<1.
\]
Thus $\Psi_n(1)>0$.

At $m=(n+3)/(3(n+1))$, we use the elementary estimate
$\sqrt{1+s}\ge1+s/2-s^2/8$ for $s\ge0$, together with
$k<7/5+6/(5n)$. After clearing positive denominators, the resulting
lower bound for $(n+k)\Psi_n$ is
\[
\frac{6n^3-53n^2+119n+262}{630n(n+1)(n-2)}.
\]
With $q=n-6$, its numerator is
$6q^3+55q^2+131q+364>0$. Hence
\[
\Psi_n\left(\frac{n+3}{3(n+1)}\right)>0.
\]

The concavity of $\Psi_n$ and the two positive endpoint values imply $\Psi_n(m)>0$ on the entire interval in the lemma.
\end{proof}

\subsection{Coefficients for the low-dimensional cases}

\begin{lemma}\label{ax:lem:A2}
The following inequalities hold:
\begin{equation}\label{ax:eq:A2}
\begin{aligned}
\gamma_3A_3(t)-\frac{2}{7}D_3(t)>0,
&\text{ for any }0\leq t\leq \frac{1}{8}\\
\gamma_4A_4(t)-\frac{5}{18}D_4(t)>0, &\text{ for any }0\leq t\leq \frac{1}{2},\\
\gamma_5A_5(t)-\frac{33}{128}D_5(t)>0, &\text{ for any }0\leq t\leq \frac{3}{4},
\end{aligned}
\end{equation}
where $\gamma_n$ is defined in Lemma \ref{ax:lem:5.4}.
\end{lemma}
\begin{proof}
Each expression in \eqref{ax:eq:A2} is a concave quadratic in
$t$, so it suffices to check the two endpoints of its stated
interval. The values of $\alpha$ satisfy
\[
 \frac58<2^{-2/3}<\frac{63}{100},\qquad
 \frac7{10}<2^{-1/2}<\frac{71}{100},\qquad
 \frac34<2^{-2/5}<\frac{379}{500}.
\]
At $t=0$, the endpoint values are
\[
\begin{aligned}
 \gamma_3A_3(0)-\frac27D_3(0)
 &=\frac{2(1-\alpha)(9-14\alpha)}{21}>0,\\
 \gamma_4A_4(0)-\frac5{18}D_4(0)
 &=\frac{(1-\alpha)(568-769\alpha)}{72(8-5\alpha)}>0,\\
 \gamma_5A_5(0)-\frac{33}{128}D_5(0)
 &=\frac{(1-\alpha)(6855-8852\alpha)}{1280(5-4\alpha)}>0.
\end{aligned}
\]
At the other endpoints, direct evaluation gives
\[
\begin{aligned}
 \gamma_3A_3\left(\frac18\right)-\frac27D_3\left(\frac18\right)
 &=\frac{1566-3975\alpha+2366\alpha^2}{1512},\\
 \gamma_4A_4\left(\frac12\right)-\frac5{18}D_4\left(\frac12\right)
 &=\frac{7872-18752\alpha+10853\alpha^2}{576(8-5\alpha)},\\
 \gamma_5A_5\left(\frac34\right)-\frac{33}{128}D_5\left(\frac34\right)
 &=\frac{55920-130143\alpha+74388\alpha^2}{5120(5-4\alpha)}.
\end{aligned}
\]
The three quadratic numerators are decreasing on the corresponding
intervals for $\alpha$. At the upper rational bounds displayed
above, their values are respectively
\[
 \frac{4077}{5000},\qquad
 \frac{290773}{10000},\qquad
 \frac{191763}{15625},
\]
which are positive. The denominators are also positive. Thus both
endpoints are positive in each dimension, and concavity proves
\eqref{ax:eq:A2} throughout the stated intervals.
\end{proof}

\begin{lemma}\label{ax:lem:A3}
Let $n=3$, $\frac{1}{8}\leq t\leq \frac{1}{4}$ and $m=(1+6t)^{-1}$. Let 
\[
 T(m)=\frac92+m+\frac{3+\alpha}{9}(1-m).
\]
and
\begin{equation*}
 F(m)
 :=
 \left(
 \frac{729+2430m+540m^2+8m^3}{32}
 \right)^{1/3}.
\end{equation*}
Then we have
\begin{equation}\label{ax:eq:E1}
 \left(5-\frac{17\alpha}{3}\right)K_3(t)
       -\frac{21-5\alpha}{36}D_3(t)>0,
\end{equation}
\begin{equation}\label{ax:eq:Flower}
 F(m)>\frac{61}{20}+\frac{77}{40}m,
\end{equation}
and
\begin{equation}\label{ax:eq:A3D3}
 \left(5-\frac{17\alpha}{3}\right)(1-m)A_3(t)
 -D_3(t)\left[T(m)-\frac{61}{20}-\frac{77}{40}m\right]>0.
\end{equation}
\end{lemma}
\begin{proof}

The left-hand side of \eqref{ax:eq:E1} is a quadratic in $t$ whose linear and quadratic
coefficients are positive. At $t=1/8$,
\[
 3456\left[
 \left(5-\frac{17\alpha}{3}\right)K_3
 -\frac{21-5\alpha}{36}D_3
 \right]
 =
 -29085\alpha^2+43113\alpha-15552>0
\]
This proves \eqref{ax:eq:E1}.

We next verify the lower bound for the function $F(m)$ used in
the interval $2/5\le m\le4/7$. Direct differentiation shows that
$F$ is strictly concave on this interval. Its endpoint values satisfy
\[
 F\left(\frac25\right)>\frac{191}{50},
 \qquad
 F\left(\frac47\right)>\frac{83}{20}.
\]

Concavity therefore gives
\[
 F(m)>
 \frac{61}{20}+\frac{77}{40}m,
 \qquad
 \frac25\le m\le\frac47.
\]

After substituting $m=(1+6t)^{-1}$ and multiplying the
left-hand side of \eqref{ax:eq:A3D3} by the positive factor
$1+6t$, we obtain the cubic polynomial
\[
\begin{aligned}
&-\frac{21(1-\alpha)}{80}
 +(1-\alpha)
 \left(\frac{118}{5}-\frac{103\alpha}{3}\right)t+
 \left(
 \frac{193}{5}-\frac{5297\alpha}{60}
 +\frac{140\alpha^2}{3}
 \right)t^2
 +
 \left(
 \frac{44\alpha^2}{3}-\frac{307\alpha}{15}
 \right)t^3.
\end{aligned}
\]
Its second derivative is negative on $1/8\le t\le1/4$.
At the two endpoints its values are
\[
 \frac{25272-64501\alpha+38780\alpha^2}{7680},
 \qquad
 \frac{1932-4814\alpha+2815\alpha^2}{240}>0
\]
for $5/8<\alpha<63/100$. Hence the cubic is
positive throughout $[1/8,1/4]$, proving \eqref{ax:eq:A3D3}.

\end{proof}

\section{Estimates involving centered probability measures}

We give the proofs of three estimates involving centered probability measures in Lemma \ref{ax:lem:5.1}, Lemma \ref{ax:lem:6.1} and Lemma \ref{ax:lem:8.1} used in
Section~\ref{sec:axial-instability}. 

\subsection{Proof of Lemma~\ref{ax:lem:5.1}}\label{app:comparison-square-root}

\begin{proof}[Proof of Lemma~\ref{ax:lem:5.1}]
When $n=4$, we have $n/(n-2)=2$, and the moment conditions give
\[
 \int_{-1}^1(1\pm x)^2\,d\mu
 =
 \int_{-1}^1(1+x)^2\,d\mu_\theta
 =
 1+m.
\]
Thus $\mathcal R_m(\theta)=\sqrt{(1+m)/2}$, proving \eqref{ax:eq:35} and all the assertions for
$n=4$. Henceforth we assume $n\ge5$, so $1<n/(n-2)<2$.

We first prove \eqref{ax:eq:33}. Set
\[
 \theta_0=\frac1m\int_{-1}^1x^3\,d\mu.
\]
By \eqref{ax:eq:24}, $|\theta_0|\le1-m$. Furthermore, we have
\begin{equation*}
\int_{-1}^1x^k\,d\mu_{\theta_0}=\int_{-1}^1x^k\,d\mu,\qquad k=0,\dots,3.
\end{equation*}

For $T>1$, let $H_T$ be the cubic polynomial satisfying
\begin{equation*}
H_T(x_{\pm}(\theta_0))=(T+x_{\pm}(\theta_0))^{n/(n-2)}, \qquad H_T'(x_{\pm}(\theta_0))=\frac{n}{n-2}(T+x_{\pm}(\theta_0))^{2/(n-2)}.
\end{equation*}

Since
\[
 \frac{d^4}{dx^4}(T+x)^{n/(n-2)}
 =
 \frac{n}{n-2}\frac{2}{n-2}\frac{4-n}{n-2}\frac{6-2n}{n-2}(T+x)^{n/(n-2)-4}>0,
\]
$(T+x)^{n/(n-2)}-H_T$ is nonnegative on $[-1,1]$. Integrating
and using the moment identities, we obtain
\[
 \int_{-1}^1(T+x)^{n/(n-2)}\,d\mu
 \ge
 \int_{-1}^1H_T\,d\mu
 =
 \int_{-1}^1H_T\,d\mu_{\theta_0}
 =
 \int_{-1}^1(T+x)^{n/(n-2)}\,d\mu_{\theta_0}.
\]
Taking $T\to 1$ by uniform convergence and taking the
$(n-2)/n$ power proves the plus sign in \eqref{ax:eq:33}.
Applying the same argument to the reflected measure proves
the minus sign, since reflection changes $\theta_0$ to
$-\theta_0$.

We next establish monotonicity and convexity. For $1<n/(n-2)<2$,
a change of variables gives the integral representation
\[
 z^{\frac{n}{n-2}}
 =
 \frac{\displaystyle
 \int_0^\infty\frac{z^2}{z+y}\,y^{\frac{n}{n-2}-2}\,dy}
 {\displaystyle
 \int_0^\infty\frac{y^{\frac{n}{n-2}-2}}{1+y}\,dy},
 \qquad z\ge0.
\]

Since $x_{\pm}(\theta)^2=\theta x_{\pm}(\theta)+m$ and $\int_{-1}^{1}xd\mu_{\theta}=0$, we obtain 
\[
 \int_{-1}^1\frac{d\mu_\theta(x)}{s+x}
 =
 \frac{s+\theta}{s^2+s\theta-m},
 \qquad s>1.
\]
 
Applying the integral representation with $z=1+x$ and using
the identity above, we obtain
\[
 \mathcal{R}(\theta)=\int_{-1}^1(1+x)^{\frac{n}{n-2}}\,d\mu_\theta(x)
 =
 \frac{\displaystyle
 \int_0^\infty
 \left[
 \frac1{1+y}
 +
 \frac{y^2m}
 {(1+y)\bigl((1+y)^2+(1+y)\theta-m\bigr)}
 \right]y^{\frac{n}{n-2}-2}\,dy}
 {\displaystyle
 \int_0^\infty\frac{y^{\frac{n}{n-2}-2}}{1+y}\,dy}.
\]
For each $y>0$, the expression in brackets is decreasing and log-convex in $\theta$. 

For any $m-1\leq \theta_1<\theta_2\leq 1-m$ and $0\leq\lambda\leq1$. For $\theta=\lambda\theta_1+(1-\lambda)\theta_2$, by H\"older's inequality, we have
\begin{equation*}
\begin{aligned}
&\int_0^\infty
 \left[
 \frac1{1+y}
 +
 \frac{y^2m}
 {(1+y)\bigl((1+y)^2+(1+y)\theta-m\bigr)}
 \right]y^{\frac{n}{n-2}-2}\,dy\\
 \leq &\left(\int_0^\infty
 \left[
 \frac1{1+y}
 +
 \frac{y^2m}
 {(1+y)\bigl((1+y)^2+(1+y)\theta_1-m\bigr)}
 \right]y^{\frac{n}{n-2}-2}\,dy\right)^{\lambda}\\
 &\cdot \left(\int_0^\infty
 \left[
 \frac1{1+y}
 +
 \frac{y^2m}
 {(1+y)\bigl((1+y)^2+(1+y)\theta_2-m\bigr)}
 \right]y^{\frac{n}{n-2}-2}\,dy\right)^{1-\lambda}.
\end{aligned}
\end{equation*}

Consequently,
$\mathcal R_m$ is non-increasing and 
\begin{equation*}
\mathcal R_m(\theta)\leq \mathcal R_m(\theta_1)^{\lambda}\mathcal R_m(\theta_2)^{1-\lambda}\leq \lambda \mathcal R_m(\theta_1)+(1-\lambda)\mathcal R_m(\theta_2).
\end{equation*}

At $\theta=0$, the measure $\mu_0=\frac{1}{2}\delta_{-\sqrt{m}}+\frac{1}{2}\delta_{\sqrt{m}}$. Define
\[
 A(y)
 =
 \frac{(1+y)^{\frac{n}{n-2}}+(1-y)^{\frac{n}{n-2}}}{2},
 \qquad 0\le y<1.
\]
Then
\[
 \mathcal R_m(0)=\alpha A(\sqrt m)^{\frac{n-2}{n}}.
\]
For $n\ge4$, we have $A(y)\ge1$ and
\begin{align*}
 \frac{d^2}{dy^2}\left(A(y)^{\frac{2(n-2)}{n}}\right)
 \ge\frac4{n-2}.
\end{align*}
Using $A(0)=1$ and $A'(0)=0$ and integrating twice, we obtain
\[
 A(y)^{\frac{2(n-2)}n}
 \ge1+\frac{2y^2}{n-2}.
\]
Setting $y=\sqrt m$ and taking square roots gives
\[
 \mathcal R_m(0)
 =
 \alpha A(\sqrt m)^{\frac{n-2}n}
 \ge
 \alpha\sqrt{1+\frac{2m}{n-2}},
\]
which proves \eqref{ax:eq:34}.

For the stronger lower bound in dimension five, take $n=5$,
so $\alpha=2^{-2/5}$. Since
\[
 (1+y)^{5/3}+(1-y)^{5/3}
 -
 \left(\frac{3+y^2}{2}\right)^{5/3}\geq 0,
 \qquad 0\le y\le1.
\]
we get
\[
 \mathcal R_m(0)
 =
 \frac12
 \left[
 (1+\sqrt m)^{5/3}+(1-\sqrt m)^{5/3}
 \right]^{3/5}
 \ge\frac{3+m}{4}.
\]

It remains to estimate $\mathcal R_m'(0)$ when $n=5$.

At $\theta=0$, we have $x_{\pm}(0)=\pm \sqrt m$ and
$x_\pm'(0)=1/2$. Differentiating the expression for
$\mathcal R_m(\theta)$ therefore gives
\[
 -\frac{\mathcal R_m'(0)}m
 =
 \frac{3\alpha A(\sqrt m)^{-2/5}}{20m}
 \left[
 \frac{(1+\sqrt m)^{5/3}-(1-\sqrt m)^{5/3}}{\sqrt m}
 -
 \frac53\bigl((1+\sqrt m)^{2/3}+(1-\sqrt m)^{2/3}\bigr)
 \right].
\]
The binomial expansion yields
\begin{align*}
 &\frac1m
 \left[
 \frac{(1+\sqrt m)^{5/3}-(1-\sqrt m)^{5/3}}{\sqrt m}
 -
 \frac53\bigl((1+\sqrt m)^{2/3}+(1-\sqrt m)^{2/3}\bigr)
 \right]=
 -4\sum_{j\ge1}j\binom{5/3}{2j+1}m^{j-1}.
\end{align*}
All coefficients on the right are positive, since
$\binom{5/3}{2j+1}<0$ for every $j\ge1$. The series is therefore
nondecreasing in $m$. Its limit as $m\to 1$, evaluated
from the expression on the left, is $\frac{2^{2/3}}3$.

Since $A(y)\ge1$, we have
\[
 0\le-\frac{\mathcal R_m'(0)}m
 \le
 \frac{3\alpha}{20}\,\frac{2^{2/3}}3
 =
 \frac{2^{4/15}}{20}
 <
 \frac3{40}.
\]
This completes the proof of \eqref{ax:eq:36}.
\end{proof}

\subsection{Proof of Lemma~\ref{ax:lem:6.1}}\label{app:comparison-affine}

\begin{proof}[Proof of Lemma~\ref{ax:lem:6.1}] 
Recall that
\[
 p=\frac{2n}{n-2}\in(2,3],
 \qquad
 \alpha=2^{2/p-1},
 \qquad n\ge6.
\]
Let $\eta$ be a probability measure on $[-1,1]$ satisfying
\[
 \int_{-1}^1x\,d\eta=0,
 \qquad
 \int_{-1}^1x^2\,d\eta=m.
\]

For $2<p\le3$, cubic Hermite interpolation gives
\begin{equation}\label{ax:eq:B1}
 (1+x)^p
 \ge
 1+px+(2^{p-1}-1)x^2+(2^{p-1}-p)x^3,
 \qquad -1\le x\le1.
\end{equation}
To justify the endpoint $x=-1$, one may first apply the interpolation
to $(T+x)^p$, $T>1$, and then let $T\to 1$. Indeed, the fourth
derivative and the factor $x^2(x^2-1)$ in the Hermite remainder have
the same sign.

Set
\[
 e=2^{p-1}-p,
 \qquad
 q=\int_{-1}^1x^3\,d\eta.
\]
By the argument for \eqref{ax:eq:24}, we have $|q|\le m(1-m)$.
Applying \eqref{ax:eq:B1} to $x$ and $-x$, integrating, and using
the concavity of the power $2/p$, we obtain
\[
 \frac{\alpha}{2}
 \left[
 \left(\int_{-1}^1(1+x)^p\,d\eta\right)^{2/p}
 +
 \left(\int_{-1}^1(1-x)^p\,d\eta\right)^{2/p}
 \right]
 \ge F_p(m),
\]
where
\[
 F_p(m)
 =
 \frac{\alpha}{2}
 \left(A(m)^{2/p}+B(m)^{2/p}\right),
\]
with
\[
 A(m)=1+(p-1+2e)m-em^2,
 \qquad
 B(m)=1+(p-1)m+em^2.
\]
Here the lower bound is obtained by replacing $|q|$ with its
largest possible value $m(1-m)$.

For $2<p\le3$,
\[
 0\le e\le1,\qquad
 p-1\ge2e,\qquad
 (p-1)^2\ge4e.
\]
These elementary inequalities imply
\[
 A\ge B\ge1,\qquad A',B'>0,
\]
and a direct differentiation gives
\begin{equation}\label{ax:eq:B2}
 F_p'''(m)>0,
 \qquad 0<m<1.
\end{equation}
For completeness, after removing the positive factor
$\alpha/[2]\,(2/p)(1-2/p)$, the derivative is
\[
 \begin{aligned}
 &(2-2/p)
 \left[
 (A')^3A^{2/p-3}+(B')^3B^{2/p-3}
 \right]\\
 &\qquad
 +6eA'A^{2/p-2}-6eB'B^{2/p-2},
 \end{aligned}
\]
and the inequalities above make this expression strictly positive.

The endpoint values are
\[
 F_p(0)=\alpha,\qquad
 F_p(1)=2,\qquad
 F_p'(1)=\frac4p\left(1-2^{1-p}\right).
\]
By \eqref{ax:eq:B2}, since $F_p'''>0$, the quadratic Hermite interpolant determined by
the value at $0$ and the value and first derivative at $1$ lies
strictly below $F_p$ on $(0,1)$. Hence
\begin{equation}\label{ax:eq:B3}
 F_p(m)>
 \alpha+(2-\alpha)m+\kappa_nm(1-m),
 \qquad 0<m<1,
\end{equation}
where
\[
 \kappa_n
 =
 2-\alpha
 -\frac4p\left(1-2^{1-p}\right)
 =
 2-\alpha
 -\frac{2(n-2)}n
 \left(1-2^{1-2n/(n-2)}\right).
\]
Together with \eqref{ax:eq:B3}, this proves the norm inequality \eqref{ax:eq:37}.

It remains to verify the coefficient gap used in
Section~\ref{sec:axial-6}. Set
\[
 s=\frac2n\in\left(0,\frac13\right],
 \qquad L=\log2.
\]
A direct substitution gives
\begin{equation}\label{ax:eq:B4}
 \kappa_n
 -
 \frac n2
 \left(
 \frac{4\alpha}{n(n-2)}
 -\frac{2(1-\alpha)}n
 \right)
 =
 \frac{\mathcal F(s)}{1-s},
\end{equation}
where
\[
 \mathcal F(s)
 =
 1+s-2s^2
 +(1-s)^2e^{-2Ls/(1-s)}
 -(2-s)e^{-Ls}.
\]
We have $\mathcal F(0)=\mathcal F'(0)=0$,
and
\[
 \mathcal F'''(s)
 =
 L^2e^{-Ls}
 \left[
 3+L(2-s)
 -
 8L\frac{e^{-2Ls/(1-s)+Ls}}{(1-s)^4}
 \right].
\]
The function $\frac{e^{-2Ls/(1-s)+Ls}}{(1-s)^4}$
is increasing on $[0,1/3]$ and equals one at $s=0$. Since
$1/2<L<3/4$, the bracket above is bounded by $3-6L<0$.
Thus $\mathcal F'''<0$ on $[0,1/3]$, so $\mathcal F''$ is
decreasing. Using $\mathcal F(0)=\mathcal F'(0)=0$, for
$0<s\le1/3$ we obtain
\[
 \frac{\mathcal F(s)}{s^2}
 =
 \int_0^1(1-\rho)\mathcal F''(s\rho)\,d\rho
 \ge
 9\mathcal F(1/3).
\]
A direct evaluation gives
\[
 9\mathcal F(1/3)
 =
 3\left(4-5\,2^{-1/3}\right)>0,
\]
where the last inequality is equivalent to
$2^{-1/3}<4/5$. Hence \(\mathcal F(s)>0\), and
\eqref{ax:eq:B4} proves
\[
 \kappa_n>
 \frac n2
 \left(
 \frac{4\alpha}{n(n-2)}
 -\frac{2(1-\alpha)}n
 \right).
\]
Finally, $ \alpha=2^{-2/n}>1-\frac2n$
implies $ \frac{4\alpha}{n(n-2)}
 -\frac{2(1-\alpha)}n>0$.
Therefore
\[
 \kappa_n>
 \frac n2
 \left(
 \frac{4\alpha}{n(n-2)}
 -\frac{2(1-\alpha)}n
 \right)>0,
\]
which completes the proof of Lemma~\ref{ax:lem:6.1}.
\end{proof}

\subsection{Proof of Lemma~\ref{ax:lem:8.1}}\label{app:comparison-translated}

\begin{proof}[Proof of Lemma~\ref{ax:lem:8.1}]
Let $x_-(\theta)<x_+(\theta)$ be the two support points of
$\mu_\theta$.
Set
\[
 y_\pm=\tau+x_\pm.
\]
Since $\tau\ge1$ and $x_\pm\in[-1,1]$, both $y_-$ and $y_+$ are
nonnegative. Let $\pi\in(0,1)$ denote the mass of $\mu_\theta$ at $x_+$, or equivalently at $y_+$. Since $\mu_\theta$ has zero first moment, the mass $\pi$ at $x_+$
is determined by
\[
 (1-\pi)x_-+\pi x_+=0,
 \qquad\text{i.e.}\qquad
 \pi=-\frac{x_-}{x_+-x_-}.
\]
Using
\[
 x_\pm=\frac{\theta\pm\sqrt{\theta^2+4m}}2
\]
and differentiating with respect to $\theta$, we obtain
\[
 y_-'=\pi,\qquad
 y_+'=1-\pi,\qquad
 \pi'=-\frac{2\pi(1-\pi)}{y_+-y_-}.
\]
Moreover, the identity $x_-x_+=-m$ gives
\[
 m=\pi(1-\pi)(y_+-y_-)^2.
\]

For $f(y)=y^6$, set
\[
 M(\theta)=(1-\pi)y_-^6+\pi y_+^6.
\]
Introduce the auxiliary function
\[
 \Psi(a,b)
 =
 120\int_0^1
 s(1-s)\bigl((1-s)a+sb\bigr)^3\,ds.
\]
Using the formulas for $y_\pm'$ and $\pi'$, together with
\[
 m=\pi(1-\pi)(y_+-y_-)^2,
\]
a direct calculation gives
\[
 M'=m\,\Psi(y_-,y_+).
\]
Since $m$ is fixed as $\theta$ varies, differentiating once more yields
\[
 M''
 =
 m\left[
 \pi\,\partial_1\Psi(y_-,y_+)
 +(1-\pi)\,\partial_2\Psi(y_-,y_+)
 \right].
\]

For brevity, write
\[
 \Psi=\Psi(y_-,y_+),\qquad
 \Psi_-=\partial_1\Psi(y_-,y_+),\qquad
 \Psi_+=\partial_2\Psi(y_-,y_+).
\]
If
\[
 y_s=(1-s)y_-+sy_+,
\]
then differentiation under the integral gives
\[
 \begin{aligned}
 \Psi=120\int_0^1s(1-s)y_s^3\,ds,\quad \Psi_-=360\int_0^1s(1-s)^2y_s^2\,ds,\quad 
 \Psi_+=360\int_0^1s^2(1-s)y_s^2\,ds.
 \end{aligned}
\]
Thus
\[
 M'=m\Psi,\qquad
 M''=m\bigl(\pi\Psi_-+(1-\pi)\Psi_+\bigr).
\]
By Cauchy--Schwarz inequality and
\[
 (y_+-y_-)^2\int_0^1(1-s)y_s^4\,ds
 =
 \int_{y_-}^{y_+}(y_+-y)y^4\,dy
 \le\frac{y_+^6}{30},
\]
we obtain
\[
 \frac23(y_+-y_-)^2\Psi^2
 \le
 \frac89y_+^6\Psi_+
 <
 y_-^6\Psi_-+y_+^6\Psi_+.
\]
Using the formulas for $M'$ and $M''$, this implies
\[
 MM''-\frac23(M')^2>0.
\]
Hence $M^{1/3}$ is strictly convex in the interior of the interval. Since $M'=m\Psi>0$, it is also increasing, which proves the corresponding assertions
for $\mathcal A_{m,\tau}$. Continuity at $|\theta|=1-m$ follows
from the explicit formula for $\mu_\theta$.

We next compare the measure $\mu$ with
$\mu_{\theta_0}$, where $\theta_0:= \frac1m\int_{-1}^1x^3\,d\mu$.
The two measures have the same moments of orders
zero through three. Let $H$ be the cubic Hermite polynomial that
agrees with $(\tau+x)^6$ and its first derivative at the two support
points of $\mu_{\theta_0}$. Since $ \frac{d^4}{dx^4}(\tau+x)^6\ge0$,
the Hermite remainder is nonnegative on $[-1,1]$. Therefore
\[
 \int_{-1}^1(\tau+x)^6\,d\mu
 \ge
 \int_{-1}^1(\tau+x)^6\,d\mu_{\theta_0}.
\]
Reflecting $x$ gives the corresponding estimate with $\tau-x$
and $\mu_{-\theta_0}$. Thus, for $0\le\lambda\le1$,
\begin{align*}
 &\alpha\biggl[
 \lambda
 \left(\int_{-1}^1(\tau+x)^6\,d\mu\right)^{1/3}
 +(1-\lambda)
 \left(\int_{-1}^1(\tau-x)^6\,d\mu\right)^{1/3}
 \biggr]\\
 &\qquad\ge
 \lambda\mathcal A_{m,\tau}(\theta_0)
 +(1-\lambda)\mathcal A_{m,\tau}(-\theta_0)\ge
 \mathcal A_{m,\tau}\bigl((2\lambda-1)\theta_0\bigr),
\end{align*}
where the last step follows from convexity.

Finally, when $\theta=0$, the measure $\mu_0$ assigns equal mass
to $\pm\sqrt m$. Since $\alpha^3=1/4$ in dimension three,
\[
 \mathcal A_{m,\tau}(0)^3
 =
 \frac14
 \left[
 \frac{(\tau+\sqrt m)^6+(\tau-\sqrt m)^6}{2}
 \right]
 =
 \frac{\tau^6+15\tau^4m+15\tau^2m^2+m^3}{4}.
\]
This proves the last assertion.
\end{proof}

\section{The weighted norm estimate in dimension three}
We prove the estimate used in \eqref{ax:eq:49} and Lemma~\ref{ax:lem:5.8} for $n=3$ and $4/7\le m<1$, equivalently $0<t\le1/8$. Recall 
\begin{equation*}
 \mathcal N_\omega=
 \frac{\alpha}{2}\left[
 (1-\omega)\|\sqrt{1+x}\,u\|_p^2
 +(1+\omega)\|\sqrt{1-x}\,u\|_p^2
 \right].
\end{equation*}
as in \eqref{ax:eq:weighted-root-norm} with 
\begin{equation*}
\omega
 =
 -\frac{\gamma_nK_n(t)\int_{-1}^1x^3\,d\mu}
 {D_n(t)L_n(m)}.
\end{equation*}
\begin{lemma}\label{ax:lem:B1}
Let $n=3$, $\frac{4}{7}\leq m<1$ and $t=\frac{1-m}{6m}$. Then we have
\begin{equation*}
N_\omega\geq 1-\frac{2}{7}(1-m).
\end{equation*}
\end{lemma}
\begin{proof}
Write
$h=1-m\in[0,3/7]$ and set
\[
 z=
 \frac{\int_{-1}^1x^3\,d\mu}{4-3h}.
\]
By \eqref{ax:eq:24},
$|z|\le h/4\le3/28$. The binomial series gives
\[
 \frac{(1+z)^{1/3}+(1-z)^{1/3}}2
 \ge1-\frac{z^2}{8},
\]
and
\[
 0\le
 \frac{(1+z)^{1/3}-(1-z)^{1/3}}{2z}
 \le\frac38.
\]

We also claim that
\[
 m\left(-\frac{K_3(t)}{D_3(t)}\right)_+
 \leq
 \frac{13}{18}\left(1-\frac{8h}{3}\right)_+,
 \qquad
 t=\frac{h}{6(1-h)},
\]
where $s_+=\max\{s,0\}$.
Indeed, for $0\leq t\leq 1/10$, the formulas of coefficients and
$\alpha>5/8$ give
\[
\begin{aligned}
 18K_3(t)+13(1-10t)D_3(t)
 &=
 16\alpha-10+(69\alpha-33)t
 +\left(\frac{836\alpha}{3}-188\right)t^2
 -\frac{260\alpha}{3}t^3\\
 &>
 \frac{t}{24}\bigl(243-332t-1300t^2\bigr)
 \geq 0.
\end{aligned}
\]
Since $D_3(t)>0$, it follows that
\[
 \left(-\frac{K_3(t)}{D_3(t)}\right)_+
 \leq \frac{13}{18}(1-10t).
\]
For $1/10\leq t\leq 1/8$, the function $K_3$ is increasing and
\[
 K_3\left(\frac{1}{10}\right)
 =\frac{1280\alpha-759}{900}>0.
\]
Thus, throughout $0\leq t\leq 1/8$,
\[
 \left(-\frac{K_3(t)}{D_3(t)}\right)_+
 \leq \frac{13}{18}(1-10t)_+.
\]
Multiplying by $m=1-h$ and using
\[
 m(1-10t)=1-\frac{8h}{3}
\]
proves the claim.

Using these estimates, the last assertion of Lemma~\ref{ax:lem:A1}, $L_3(m)>m$ from Lemma~\ref{ax:lem:5.4}, and the
third-moment bound \eqref{ax:eq:24}, the weighted norm in
\eqref{ax:eq:49} is bounded below by
\[
 \Phi(h)
 =
 \left(1-\frac{3h}{4}\right)^{1/3}
 -\frac{h^2(1-h)^2}{8(4-3h)^2}
 -\frac{13h^2}{288(4-3h)}
 \left(1-\frac{8h}{3}\right)_+,
\]
where $s_+=\max\{s,0\}$.

Differentiating on the two intervals separated by $h=3/8$
shows that $\Phi$ is concave; at the joining point its derivative
can only decrease. The endpoint values satisfy
\[
 \Phi(0)=1,\qquad
 \Phi\left(\frac37\right)
 =\left(\frac{19}{28}\right)^{1/3}-\frac{18}{17689}
 >\frac{43}{49}.
\]
The last inequality follows by cubing the two positive sides after
moving the rational term to the right. Concavity therefore gives
\[
 \Phi(h)\ge1-\frac{2h}{7},
 \qquad
 0\le h\le\frac37.
\]
Since $h=1-m$, the desired estimate follows.    
\end{proof}

\bibliographystyle{amsplain}
\bibliography{references}
\end{document}